\documentclass[preprint,3p,sort&compress]{elsarticle}

\journal{{\tt arXiv.org}}

\usepackage[dvips]{epsfig}
\usepackage{graphicx} 
\usepackage{latexsym}
\usepackage{verbatim}
\usepackage{amsmath}
\usepackage{amsthm}
\usepackage{amssymb}
 \usepackage[hang,footnotesize,bf]{caption}
\usepackage{subfig}
\usepackage{bbm}
\usepackage{pstool}
\usepackage[boxed]{algorithm2e}
\usepackage{mathtools}
\usepackage{fonts}
\usepackage{defs}
\usepackage{placeins}
\usepackage[scientific-notation=true]{siunitx}
\definecolor{sepia}         {cmyk}{0   , 0.83, 1   , 0.70}
\definecolor{maroon}        {cmyk}{0   , 0.87, 0.68, 0.32}

\usepackage{booktabs} 
\usepackage{microtype}
\newcommand{\InsertFig}[3]{
  \begin{figure}[!htbp]
    \begin{center}
      \leavevmode
      #1
      \caption{#2}
      #3
    \end{center}
  \end{figure}
}

\newif\iflongpaper
	\longpaperfalse 
\newcommand{\longpaper}[2]{\iflongpaper #1 \else #2 \fi}

\newtheorem{theorem}{Theorem}[section]
\newtheorem{definition}[theorem]{Definition}
\newtheorem{remark}[theorem]{Remark}

\newtheorem{proposition}[theorem]{Proposition}
\newtheorem{lemma}[theorem]{Lemma}

\newcommand{\floor}[1]{\ensuremath{\left\lfloor #1 \right\rfloor}}

\newcommand{\secref}[1]{Section~\ref{#1}}
\newcommand{\thmref}[1]{Theorem~\ref{#1}}

\newcommand{\lemref}[1]{Lemma~\ref{#1}}

\newcommand{\figref}[1]{Figure~\ref{#1}}
\newcommand{\tabref}[1]{Table~\ref{#1}}

\newcommand{\abs}[1]{\ensuremath{\left| #1 \right|}}

\newcommand{\eps}{\varepsilon}

\newcommand{\R}{\mathbb{R}}

\newcommand{\collconst}{{c}}

\def\xL{x_{\text{L}}}
\def\xR{x_{\text{R}}}
\def\psiL{\psi_{\text{L}}}
\def\psiR{\psi_{\text{R}}}
\def\tf{t_{\rm f}}
\def\psiInit{\psi_{t = 0}}
\def\etad{\eta_{\ast}}

\newcommand{\U}{\ensuremath{\mathbf{u}}}

\newcommand{\Uzero}{\ensuremath{\U_{t = 0}}}

\newcommand{\Uhat}{\ensuremath{\hat{\U}}}
\newcommand{\Ubar}{\ensuremath{\bar{\U}}}
\newcommand{\UcollR}{\ensuremath{\U_{\cC}}}

\newcommand{\psip}[1]{\psi^{\left( #1\right)}_+}
\newcommand{\psim}[1]{\psi^{\left( #1\right)}_-}

\newcommand{\RD}[2]{\ensuremath{\mathcal{R}_{#1}^{#2}}}
\newcommand{\RQ}[1]{\RD{#1}{\mathcal{Q}}}

\newcommand{\RQone}[1]{\left. \RQ{#1} \right|_{u_0 = 1}}
\newcommand{\RQsone}[1]{\left. \RQ{#1} \right|_{u_0 < 1}}

\newcommand{\source}{S}
\newcommand{\nmom}{\ensuremath{N}}
\newcommand{\nqmu}{\ensuremath{{n_\mathcal{Q}}}}
\newcommand{\nqs}{\ensuremath{{Q}}}
\newcommand{\ncells}{\ensuremath{{J}}}
\newcommand{\ntime}{\ensuremath{{N_t}}}
\newcommand{\order}{\ensuremath{k}}
\newcommand{\timelvl}{\ensuremath{n}}

\newcommand{\basisComp}{\ensuremath{b}}
\newcommand{\basis}{\ensuremath{ \mathbf{\basisComp}}}

\newcommand{\mbasisComp}{\ensuremath{p}}
\newcommand{\mbasis}{\ensuremath{ \mathbf{\mbasisComp}}}
\newcommand{\mmbasisComp}{\ensuremath{m}}
\newcommand{\mmbasis}{\ensuremath{ \mathbf{\mmbasisComp}}}

\def\RQb{\RQ{\basis}}
\def\RQm{\RQ{\mbasis}}
\def\RQmm{\RQ{\mmbasis}}

\def\RQbOne{\overline{\RQb}|_{u_0 = 1}}
\def\RQmOne{\overline{\RQm}|_{u_0 = 1}}
\def\RQmmOne{\overline{\RQmm}|_{u_0 = 1}}

\def\RQmmOneBig{\left. \overline{\RQmm}\right|_{u_0 = 1}}

\def\RQblOne{\overline{\RQb}|_{u_0 \le 1}}

\def\RQblOneBig{\left. \overline{\RQb}\right|_{u_0 \le 1}}

\newcommand{\MN}{\ensuremath{\text{M}_{\nmom}}}
\newcommand{\MMN}{\ensuremath{\text{MM}_{\nmom}}}

\newcommand{\bfM} { \mbox{\boldmath $M$} }

\def\alphahat{\hat{\bsalpha}}

\def\uIso{\bu_{\text{iso}}}

\newcommand{\collision}[1]{\ensuremath{\cC\left(#1\right)}}
\newcommand{\collisionop}{\ensuremath{\cC}}

\def\ansatz{\hat \psi}
\def\ansatzu{\ansatz_{\U}}

\DeclareMathOperator*{\argmin}{argmin}
\DeclareMathOperator{\interior}{int}
\DeclareMathOperator{\co}{co}

\numberwithin{equation}{section}

\newif\ifexternalizing 
	\externalizingfalse
	
\ifexternalizing
     \usepackage{tikz}
     \usepackage{pgfplots}
     \definecolor{darkgreen}{RGB}{0, 128,0}
     \usepgfplotslibrary{external}
     \pgfplotscreateplotcyclelist{color mod}{%
         blue!30,every mark/.append style={fill=blue!30!black},mark=*\\%
         red!60,every mark/.append style={fill=red!30!black},mark=square*\\%
         brown!60!black,every mark/.append style={fill=brown!80!black},
         mark=diamond\\%
         black,mark=triangle\\%
         blue,every mark/.append style={fill=blue!80!black},mark=diamond*\\%
         red,densely dashed,every mark/.append style={solid,fill=red!80!black},
         mark=*\\%
         brown!60!black,densely dashed,every mark/.append style={
         solid,fill=brown!80!black},mark=square*\\%
         black,densely dashed,every mark/.append style={solid,fill=gray},
         mark=otimes*\\%
         blue,densely dashed,mark=star,every mark/.append style=solid\\%
         red,densely dashed,every mark/.append style={solid,fill=red!80!black},
         mark=diamond*\\%
     }
     \pgfplotscreateplotcyclelist{color mod line}{%
     {color=blue},{color=black},{color=red},{color=black,dashed},
     {color=orange,dashdotted},{color=blue,dotted}} 
\else
\fi

\title{A realizability-preserving discontinuous Galerkin scheme for 
entropy-based moment closures for linear kinetic equations in one space
dimension}

\author[ga]{Graham Alldredge}
\author[fs]{Florian Schneider}
\address[ga]{Department of Mathematics, RWTH Aachen University, Schinkelstr. 2,52062 Aachen, Germany, {\tt alldredge@mathcces.rwth-aachen.de}}
\address[fs]{Fachbereich Mathematik, TU Kaiserslautern, Erwin-Schr\"odinger-Str., 67663 Kaiserslautern, Germany, {\tt fschneid@mathematik.uni-kl.de}}

\begin{document}
\begin{abstract}
We implement a high-order numerical scheme for the entropy-based moment
closure, the so-called M$_{\nmom}$ model, for linear
kinetic equations in slab geometry.
A discontinuous Galerkin (DG) scheme in space along with a strong-stability
preserving Runge-Kutta time integrator is a natural choice to achieve a
third-order scheme, but so far, the challenge for such a scheme in this
context is the implementation of a linear scaling limiter when the numerical
solution leaves the set of realizable moments (that is, those moments
associated with a positive underlying distribution).
The difficulty for such a limiter lies in the computation of the intersection
of a ray with the set of realizable moments.
We avoid this computation by using quadrature to generate a convex polytope
which approximates this set.
The half-space representation of this polytope is used to compute an
approximation of the
required intersection straightforwardly, and
with this limiter in hand, the rest of the DG scheme is constructed
using standard techniques.
We consider the resulting numerical scheme on a new manufactured solution
and standard benchmark problems for both traditional M$_{\nmom}$ models
and the so-called mixed-moment models.
The manufactured solution allows us to observe the expected convergence
rates and explore the effects of the regularization in the optimization.
\end{abstract}
\noindent
\maketitle
\ifpdf
    \graphicspath{{Images/PNG/}{Images/PDF/}{Images/}}
\else
    \graphicspath{{Chapter4/Chapter4Figs/EPS/}{Chapter4/Chapter4Figs/}}
\fi


\section{Introduction}
Moment closures are a class of spectral methods used in the context of
kinetic transport equations.
An infinite set of moment equations is defined by taking velocity-
or phase-space averages with respect to some basis of the velocity space. A
reduced description of the kinetic density is then
achieved by truncating this hierarchy of equations at some finite order.
The remaining equations however inevitably require information from the
equations which were removed.
The specification of this information, the so-called moment closure problem, distinguishes different moment
methods.
In the context of linear radiative transport, the standard spectral method
is commonly referred to as the P$_{\nmom}$ closure \cite{Lewis-Miller-1984},
where $\nmom$ is the order of the highest-order moments in the model.
The P$_{\nmom}$ method is powerful and simple to implement, but does not
take into account the fact that the original function to be
approximated, the kinetic density, must be non-negative.
Thus P$_{\nmom}$ solutions can contain negative values for the local
densities of particles, rendering the solution physically meaningless.

Entropy-based moment closures, referred to as M$_{\nmom}$ models in the
context of radiative transport \cite{Min78,DubFeu99},
have all the properties one would desire in a moment method, namely
positivity of the underlying kinetic density,%
\footnote{
Positivity is actually not gained for every entropy-based moment closure
but is indeed a property of those models derived from important, physically
relevant entropies.
}
hyperbolicity of the closed system of equations,
and entropy dissipation \cite{Lev96}.
Practical implementation of these models has been traditionally considered
too expensive because they require the numerical solution of an
optimization problem at every point on the space-time grid, but recently
there has been renewed interest in the models due to their inherent
parallelizability \cite{Hauck2010}.
However, while their parallelizability goes a long way in making M$_{\nmom}$
models computationally competitive, in order to make these methods truly
competitive with more basic discretizations, the gains in efficiency that
come from higher-order methods will likely be necessary.
Here the issue of realizability becomes a stumbling block.

The property of positivity implies that the system of moment
equations only evolves on the set of so-called realizable moments.
Realizable moments are simply those moments associated with positive
densities, and the set of these moments forms a convex cone which is a
strict subset of all moment vectors.
This property, while indeed desirable since it is consistent with the
original kinetic density, can cause problems for numerical methods.
Standard high-order numerical solutions to the Euler equations, which indeed are an
entropy-based moment closure, have been observed to have negative
local densities and pressures \cite{Zhang2010}.
This is exactly loss of realizability.

A recently popular high-order method for hyperbolic systems is the
Runge-Kutta discontinuous Galerkin (RKDG) method
\cite{CockburnShuPk,Cockburn1989}.
An RKDG method for moment closures can handle the loss of realizability
through the use of a realizability (or ``positivity-preserving'') limiter
\cite{Zhang2010}, but so far these have been implemented for low-order
moment systems (that is $\nmom = 1$ or $2$) \cite{Olbrant2012}
because here one can rely on the
simplicity of the structure of the realizable set for low-order moments.
For higher-order moments, the realizable set has complex nonlinear
boundaries: when the velocity domain is one-dimensional, the realizable
set is characterized by the positive-definiteness of Hankel matrices
\cite{Shohat-Tamarkin-1943,CurFial91}; in higher dimensions, the realizable
set is not well-understood.
In \cite{ahot2013}, however, the authors noticed that a quadrature-based
approximation of the realizable set is a convex polytope.
With this simpler form, one can now actually generalize the realizability
limiters of \cite{Zhang2010,Olbrant2012} for moment systems of (in
principle) arbitrary moment order.
Furthermore, this approximation of the realizable set holds in any
dimension.

In this work we begin by reviewing our kinetic equation, its 
entropy-based moment closure, and the concept of realizability in
\secref{sec:momapprox}.
Then in \secref{sec:dg} we outline how we apply the Runge-Kutta
discontinuous Galerkin scheme to the moment equations.  Here the key
ingredients are a strong-stability preserving Runge-Kutta method, a
numerical optimization algorithm to compute the flux terms, a slope
limiter, a realizability-preserving property for the cell means, and the
realizability limiter.
In \secref{sec:SIM} we present numerical results using a manufactured
solution to perform a convergence test, as well as simulations of standard
benchmark problems.
Finally in \secref{sec:Conclusions}, we draw conclusions and suggest
directions for future work.

\section{A linear kinetic equation and moment closures}
\label{sec:momapprox}

We begin with the linear kinetic equation we will use to test our algorithm
and a brief introduction to entropy-based moment closures and the concept of
realizability.
More background can be found for example in \cite{Lewis-Miller-1984,
Lev96,Hauck2010} and references therein.

\subsection{A linear kinetic equation}

We consider the following one-dimensional linear kinetic equation for the 
kinetic density $\psi = \psi(t, x, \mu) \ge 0$ in slab geometry,
for time $t > 0$, spatial coordinate
$x \in X = (\xL, \xR) \subseteq \R$, and angle variable
$\mu \in [-1, 1]$:
\begin{align}
\label{eq:FokkerPlanck1D}
\partial_t\psi + \mu \partial_x  \psi + \sig{a} \psi = 
\sig{s}\collision{\psi} + \source,
\end{align}
where $\sig{a}$ are $\sig{s}$ are the absorption and scattering interaction 
coefficients, respectively, which throughout the paper we assume for 
simplicity to be constants,%
\footnote{
All results here can be generalized to spatially inhomogeneous interaction
coefficients.
}
and $\source$ a source.
The operator $\collisionop$ is a collision operator, which in this paper we
assume to be linear and have the form
\begin{equation}
 \collision{\psi} = \int_{-1}^1 T(\mu, \mu^\prime)
  \psi(t, x, \mu^\prime)~d\mu^\prime 
  - \int_{-1}^1 T(\mu^\prime, \mu) \psi(t, x, \mu)~d\mu^\prime.
\label{eq:collisionOperator}
\end{equation}
We assume that the kernel $T$ is strictly positive and normalized to 
$\int_{-1}^1 T(\mu^\prime, \mu) d\mu^\prime~\equiv~1$.  A typical example is 
isotropic scattering, where $T(\mu, \mu^\prime) \equiv 1/2$.

Equation \eqref{eq:FokkerPlanck1D} is supplemented by initial and boundary 
conditions:
\begin{subequations}
\label{eq:bc-ic}
\begin{align}
 \psi(t, \xL, \mu) &= \psiL(t, \mu) \,, & t &\geq 0 \,, & \mu &> 0\,, 
  \label{eq:bcL} \\
 \psi(t, \xR, \mu) &= \psiR(t, \mu) \,, & t &\geq 0 \,, & \mu &< 0\,, 
  \label{eq:bcR} \\
 \psi(0, x, \mu) &= \psiInit(x, \mu)  \,, & x &\in 
  (\xL,\xR) \,, & \mu &\in [-1,1] \,,
\end{align}
\end{subequations}
where $\psiL$, $\psiR$, and $\psiInit$ are given.

\subsection{Moment equations and entropy-based closures}

Moments are defined by angular averages against a set of basis functions.  
We use the following notation for angular integrals:
$$
 \Vint{\phi} = \int_{-1}^1 \phi(\mu) d\mu
$$
for any integrable function $\phi = \phi(\mu)$;
and therefore if we collect the basis functions into a vector $\basis = 
\basis(\mu) = (\basisComp_0(\mu), \basisComp_1(\mu), \ldots, 
\basisComp_{\nmom}(\mu))^T$, then the moments of a kinetic density
$\phi = \phi(\mu)$ are given by $\U = \Vint{\basis \phi}$.

A system of partial differential equations for moments $\U = \U(t, x)$
approximating the moments $\vint{\basis \psi}$ (for the $\psi$ which satisfies
\eqref{eq:FokkerPlanck1D}) can be obtained by multiplying
\eqref{eq:FokkerPlanck1D} by $\basis$, integrating over $\mu$, and closing the
resulting system of equations by replacing $\psi$ where necessary with an
\textit{ansatz} $\ansatzu$ which satisfies $\bu = \vint{\basis \ansatzu}$.
The resulting system has the form \cite{Lev96,Hauck2010}
\begin{equation}
 \partial_t \U + \partial_x \bff(\U) + \sig{a}\U = \sig{s} \br(\U) + 
  \Vint{\basis \source},
\label{eq:moment-closure}
\end{equation}
where
$$
 \bff(\U) := \Vint{\mu \basis \ansatz_{\U}} \quand
 \br(\U)  := \Vint{\basis \collision{\ansatz_{\U}}}.
$$
In an entropy-based closure (commonly referred in standard polynomial 
bases as the $\MN$ model), the ansatz is the solution to the constrained
optimization problem
\begin{equation}
 \ansatz_{\U} = \argmin\limits_\phi \left\{\Vint{\eta(\phi)}
 : \Vint{\basis \phi} = \U \right\},
\label{eq:primal}
\end{equation}
where the kinetic entropy density $\eta$ is strictly convex and
the minimum is simply taken over functions $\phi = \phi(\mu)$ such that 
$\Vint{\eta(\phi)}$ is well defined.
The optimization problem \eqref{eq:primal} is typically numerically solved
through its strictly convex finite-dimensional dual,
\begin{equation}
 \alphahat(\U) := \argmin_{\bsalpha \in \R^{\nmom + 1}} \Vint{\eta_*(\basis^T 
  \bsalpha)} - \U^T \bsalpha,
\label{eq:dual}
\end{equation}
where $\eta_*$ is the Legendre dual of $\eta$.
The first-order necessary conditions for the dual problem 
show that the solution to the primal problem \eqref{eq:primal} has the form
\begin{equation}
 \ansatz_{\U} = \etad' \left(\basis^T \alphahat(\U) \right)
\label{eq:psiME}
\end{equation}
where $\etad'$ is the derivative of $\eta_*$.

The entropy $\eta$ can be chosen according to the  physics being modeled.
As in \cite{Hauck2010} we use Maxwell-Boltzmann entropy%
\footnote{
Indeed in a linear setting such as ours, any convex entropy $\eta$ is 
dissipated by \eqref{eq:FokkerPlanck1D}, so we have some freedom.
We focus on the Maxwell-Boltzmann entropy because it is physically relevant for 
many problems, gives a positive ansatz $\ansatz_{\U}$, and also allows
us to explore 
some of the challenges of numerically simulating entropy-based moment closures.%
} %
\begin{align}
\label{eq:EntropyM}
 \eta(z) = z \log(z) - z.
\end{align}
For the Maxwell-Boltzmann entropy \eqref{eq:EntropyM}, $\etad(y) = \etad'(y)
= \exp(y)$.

In this paper we consider both the monomial moments, defined by the basis
$$
 \mbasis(\mu) := \left(1, \mu, \ldots, \mu^{\nmom} \right)^T
$$
and the so-called mixed-moments \cite{SchneiderAlldredge14}, which contain
the usual zeroth-order moment but half moments in the higher orders.
This is achieved using the basis functions
$$
\mmbasis(\mu) := \left( 1,
\mu_+,
\mu_-, 
\ldots , \mu^\nmom_+, \mu^\nmom_- 
\right)^T,
$$
where $\mu_+ = \max(\mu, 0)$ and $\mu_- = \min(\mu, 0)$.%
\footnote{
Notice that in the mixed-moment case, there are $2\nmom + 1$ basis functions 
instead of $\nmom + 1$ as in the monomial case.
However, for clarity of exposition, for most of the paper we will assume 
$\basis$ has $\nmom + 1$ components, though everything applies to the 
mixed-moment case as well.%
} %
Mixed-moment models, which we refer to as MM$_{\nmom}$, have been
introduced to address disadvantages like 
the zero net-flux-problem and unphysical shocks in full-moment models
\cite{Frank07,SchneiderAlldredge14}.

We close this section by quickly noting that the classical P$_{\nmom}$ 
approximation \cite{Lewis-Miller-1984} is an 
entropy-based moment closure by choosing the basis $\basis$ as the Legendre 
polynomials and using the entropy
$$
 \eta(z) = \frac12 z^2.
$$
This results in the ansatz
$$
 \ansatz_{\U} = \basis^T \bsalpha,
$$
which clearly is not necessarily positive.  Nonetheless the resulting moment 
system is linear, simple to compute, and for high values of $\nmom$ provides 
good baseline solutions to the original kinetic equation
\eqref{eq:FokkerPlanck1D}.

\subsection{Moment realizability}

Since the underlying kinetic density we are trying to approximate is
nonnegative, a 
moment vector only makes sense physically if it can be associated with a 
nonnegative density. In this case the moment vector is called 
\textit{realizable}.
Additionally, since the entropy ansatz has the form \eqref{eq:psiME}, in the 
Maxwell-Boltzmann case the optimization problem \eqref{eq:primal} only has a 
solution if the moment vector lies in the ansatz space
$$
 \cA := \left\{ \Vint{\basis \exp \left( \basis^T \bsalpha \right)} 
  : \bsalpha \in \bbR^{N + 1}  \right\}.
$$
In our case, where the domain of angular integration is bounded, the ansatz 
space $\cA$ is exactly equal to the set of realizable moment vectors 
\cite{Jun00}.
Therefore we can focus simply on realizable moments, so in this section we 
quickly review their characterization in the cases of exact and approximate 
integration.

\subsubsection{Classical theory}\hskip 1pt\\
\label{sec:realizability-thy}

\begin{definition}
The \emph{realizable set} \RD{\basis}{} is 
$$
\RD{\basis}{} = \left\{\U~:~\exists \phi(\mu)\ge 0,\, \Vint{\phi} > 0,
 \text{ such that } \U = \Vint{\basis\phi} \right\}.
$$
Any $\phi$ such that $\U = \vint{\basis \phi}$ is called a \emph{representing 
density}.
\end{definition}

The realizable set is a convex cone.

In the monomial basis $\basis = \mbasis$, a moment vector is realizable if 
and only if its corresponding Hankel matrices are positive definite 
\cite{Shohat-Tamarkin-1943}.
When a moment vector sits exactly on $\partial \RD{\mbasis}{}$, there is only 
one representing density, and it is a linear combination of point masses 
\cite{CurFial91}.
In this case, the corresponding Hankel matrices are singular.
This also causes the optimization problem to be arbitrarily poorly conditioned 
as the moment vector approaches $\partial \RD{\mbasis}{}$
\cite{AllHau12}.

Realizability conditions in the mixed-moment basis $\basis = \mmbasis$ are given 
in \cite{SchneiderAlldredge14} again using Hankel matrices for each half-interval $[-1, 0]$ and $[0, 1]$ as well as another condition to ``glue'' the 
half-interval conditions together.  In this case, only a subset of the moment 
vectors on $\partial \RD{\mmbasis}{}$ have unique representing densities, but 
those that do include point masses.

\subsubsection{The numerically realizable set}
\label{sec:NumQuad}

In general, angular integrals cannot be computed analytically.  We define a 
quadrature for functions $\phi:[-1,1]\to\R$ by nodes $\{\mu_i\}_{i=1}^{\nqmu}$ and 
weights 
$\{w_i\}_{i=1}^{\nqmu}$ such that
\begin{align*}
\sum\limits_{i=1}^{\nqmu} w_i \phi(\mu_i) \approx \Vint{\phi}
\end{align*}
Below we often abuse notation and write $\Vint{\phi}$ when in 
implementation we mean its approximation by quadrature.
Then, as defined in \cite{ahot2013}, the numerically realizable set is
$$
\RQ{\basis} = \left\{\U~:~\exists f_i > 0 \text{ s.t. } \U = \sum_{i = 
1}^{\nqmu}w_i \basis(\mu_i) f_i \right\} 
$$
Indeed, when replacing the integrals in the optimization problem \eqref{eq:primal} 
with quadrature, a minimizer can only exist when $\U \in \RQ{\basis}$.
It is straightforward to show that, as expected, $\RQ{\basis} \subseteq 
\RD{\basis}{}$.\\

The numerically realizable set $\RQ{\basis}$ is the convex cone generated by
$\RQone{\basis}$, the set of normalized moment vectors:
$$
\RQone{\basis} := \left\{ \U = (u_0, u_1, \ldots , u_{\nmom})^T \in
\RQ{\basis} : u_0 = 1 \right\};
$$
and $\RQone{\basis}$ is the interior of a convex polytope:
\begin{proposition}[\cite{ahot2013}]
\label{prop:RQ}
For any quadrature $\cQ$ with positive weights $w_i$, and for simplicity 
assuming $\basisComp_0(\mu) \equiv 1$,
\begin{align*}
\RQone{\basis} = 
\interior \co \left\{\basis(\mu_i)\right\}_{i=1}^{\nqmu}
\end{align*}
where $\interior$ indicates the interior and $\co$ 
indicates the convex hull.
\end{proposition}

\section{Realizability-preserving discontinuous Galerkin scheme}
\label{sec:dg}

In this section we introduce our high-order numerical method to simulate the 
moment system \eqref{eq:moment-closure}.  We use the Runge-Kutta 
discontinuous Galerkin (RKDG) approach \cite{CockburnShuPk,Cockburn1989}
and recent techniques for the 
numerical solution of the defining optimization problem \cite{ahot2013}
\eqref{eq:primal}-\eqref{eq:dual}.  Finally in this section we discuss the 
crucial issue of realizability and our linear scaling limiter to handle 
non-realizable moments in the solution.

\subsection{The discontinuous Galerkin formulation}
\label{sec:DGScheme}
We briefly recall the discontinuous Galerkin method for a general system with 
source term:
\begin{align}
\label{eq:AbbrvSystem}
\partial_t \U + \partial_x \bff(\U) &= \bs(\U),
\end{align}
where in our case $\bs(\U) := \sig{s}\br(\U) - \sig{a}\U + \Vint{\basis 
\source}$.
We follow the approach outlined in a series of papers by Cockburn and Shu
\cite{Cockburn1989,CockburnShuPk,CockburnShuP1}.
We divide the spatial domain $(\xL, \xR)$ into $\ncells$ cells $I_j = (x_{j - 
1/2}, x_{j + 1/2})$, where the cell edges are given by $x_{j \pm 1/2} = x_j \pm 
\dx / 2$ for cell centers $x_j = \xL + (j - 1 / 2)\dx$, and
$\dx = (\xR - \xL) /  \ncells$. 
For each $t$, we seek approximate solutions $\U_h(t, x)$ in the finite
element space
\begin{equation}
V_h^k = \{v \in L^1(\xL, \xR): v|_{I_{j}} \in P^k(I_j) \text{ for } 
 j \in \{ 1, \ldots , \ncells \} \}.
\label{eq:dg-space}
\end{equation}
where $P^{k}(I)$ is the set of polynomials of degree at most $k$ on the
interval $I$.
We follow the Galerkin approach: replace $\U$ in \eqref{eq:AbbrvSystem} by a 
solution of the form $\U_h \in V_h^k$ then multiply the resulting equation
by basis 
functions $v_h$ of $V_h^k$ and integrate over cell $I_j$ to obtain:
\begin{subequations}
\label{eq:dweakform1}
\begin{align}
 \partial_t \int_{I_j} \U_h(t, x)v_h(x)\,dx
 &+ \bff(\U_h(t, x_{j + 1/2}^-)) v_h(x_{j + 1/2}^-)
  - \bff(\U_h(t, x_{j - 1/2}^+)) v_h(x_{j - 1/2}^+) \nonumber \\
 &-\int_{I_j} \bff(\U_h(t, x)) \partial_x v_h(x)\,dx 
 = \int_{I_j} \bs(\U_h(t, x))v_h(x)\,dx \\
 \label{eq:dweakform1a}
 \int_{I_j} \U_h(0, x)v_h(x)\,dx &= \int_{I_j} \Uzero(x) 
v_h(x)\,dx 
\end{align}
\end{subequations}
where $x_{j \pm 1/2}^-$ and $x_{j \pm 1/2}^+$ denote the limits 
from left and right, respectively, and $\Uzero$ is the initial 
condition.
In order to approximately solve the Riemann problem at the cell-interfaces,
the fluxes $\bff(\U_h(t, x_{j 
+ 1/2}^\pm))$ at the points of discontinuity are both replaced by a 
numerical flux $\hat \bff(\U_h(t, x_{j + 1/2}^-), \U_h(t, x_{j + 1/2}^+))$, thus 
coupling the elements with their neighbors \cite{toro2009riemann}.
In this paper we use the global Lax-Friedrichs flux:%
\footnote{
The local Lax-Friedrichs flux could be used instead.
This would require computing the eigenvalues of the Jacobian in every
space-time cell to adjust the value of the numerical viscosity constant $C$
but would possibly decrease the overall diffusivity of the scheme.
However, since we are considering high-order methods, the decrease in
diffusivity achieved by switching to the local Lax-Fridrich flux should be
negligible.%
}
\begin{align*}
 \hat \bff(\bv, \bw) = \dfrac{1}{2} \left( \bff(\bv) + \bff(\bw) - C (\bw - 
\bv) \right),
\end{align*}
The numerical viscosity constant $C$ is taken as the global estimate of the 
absolute value of the largest eigenvalue of the Jacobian
$\partial \bff / \partial \U$.
We use $C = 1$, because
for the moment systems used here it is known that the largest eigenvalue is 
bounded by one in absolute value:
\begin{lemma}
\label{lem:eigenvalues}
The eigenvalues of the Jacobian $\partial \bff / \partial \U$ are bounded
in absolute value by one.
\end{lemma}
\begin{proof}
For convenience, we present a slight generalization of Lemma 4 in 
\cite{Olbrant2012}.

We define $\bJ(\bsalpha) := \Vint{\mu \basis \basis^T \etad''(\basis^T 
\bsalpha)}$ and $\bH(\bsalpha):= \Vint{\basis \basis^T \etad''(\basis^T 
\bsalpha)}$.
Using the properties of $\bH(\bsalpha)$  given in \cite{Lev96},
by applying the chain rule we have
\begin{equation}
 \frac{\partial \bff (\U)}{\partial \U} = \bJ(\alphahat(\U))
 \frac{\partial \alphahat(\U)}{\partial \U}
  = \bJ(\alphahat(\U)) \bH(\alphahat(\U))^{-1}.
\label{eq:fluxJacobian}
\end{equation}
If $\bJ(\alphahat(\U)) \bH(\alphahat(\U))^{-1} \bc = \lambda \bc$ for some 
$\bc \ne 0$, then for $\bd = \bH(\alphahat(\U))^{-1} \bc$ we also have 
$\bJ(\alphahat(\U)) \bd = \lambda \bH(\alphahat(\U)) \bd$, and thus
$$
 |\lambda| \le \frac{\bd^T \bJ(\alphahat(\U)) \bd}{\bd^T \bH(\alphahat(\U)) \bd}
  = \frac{\Vint{\mu \left( \basis^T \bd \right)^2 
  \etad''(\alphahat(\U))}}
  {\Vint{\left( \basis^T \bd \right)^2 \etad''(\alphahat(\U))}} \le 1.
$$
The last inequality follows from the facts that $|\mu| \le 1$ and that 
$\etad'' > 0$, the latter of which is a consequence of the strict 
convexity of $\eta$.
\end{proof}

On each interval, the DG approximate solution $\U_h$ can be written as 
\begin{align}
\label{eqn:solution_form}
\U_h|_{I_j}(t, x) := \U_j(t, x) := \sum_{i=0}^{\order}\Uhat_j^i(t) 
\varphi_i\left( \frac{x - x_j}{\dx} \right)
\end{align}
where $\{\varphi_0, \varphi_1, \ldots ,\varphi_\order \}$ denote a 
basis for $P^k([-1/2, 1/2])$.
It is convenient to choose an orthogonal basis, so we use Legendre 
polynomials scaled to the interval $[-1/2, 1/2]$:
\begin{align*}
\varphi_0(y) = 1, \quad \varphi_1(y) = 2y, \quad \varphi_2(y) = 
 \frac12 (12y^2-1), \: \ldots
\end{align*}
With an orthogonal basis the cell means $\Ubar_j$ are easily available from the 
expansion coefficients $\Uhat_j$:
$$
\Ubar_j(t) := \frac{1}{\dx} \int_{I_j} \U_j(t, x) dx
 = \frac{1}{\dx} \sum_{i=0}^{\order} \Uhat_j^i(t)
 \int_{I_j} \varphi_i\left( \frac{x - x_j}{\dx} \right) dx
 = \Uhat_j^0(t)
$$
We collect the coefficients $\Uhat_j^{(i)}(t)$ into the $(\order + 1) \times 
(\nmom + 1)$ matrix
$$
\Uhat_j(t) = \left( \begin{array}{c}
           \left( \Uhat^0_j(t) \right)^T \\
           \vdots \\
           \left( \Uhat^{\order}_j(t) \right)^T
          \end{array} \right)
$$

Using the form of the approximate solution in (\ref{eqn:solution_form}), we can 
write \eqref{eq:dweakform1} in matrix form:
\begin{subequations}
\label{eq:dweakform2}
\begin{align}
\bM \partial_t \Uhat_j + 
\bF(\Uhat_{j - 1}, \Uhat_j, \Uhat_{j + 1}) - 
\bV(\Uhat_j) &= \bS(\Uhat_j) \label{eq:weakform_m1} \\
\left(\bM \Uhat_j(0)\right)_{i\ell} &= 
\int_{I_j} u_{\ell, t = 0}(x) \varphi_i\left( \frac{x - x_j}{\dx} 
\right)  dx \label{eq:weakform_m2}
\end{align}
\end{subequations}
for $j \in \{ 1, 2, \ldots , \ncells \}$, with
\begin{subequations}
\label{eq:weakFormDefs}
\begin{align}
 (\bfM)_{i\ell} =& \int_{I_j} \varphi_i \left( \frac{x - x_j}{\dx} \right)
  \varphi_\ell \left( \frac{x - x_j}{\dx} \right)~dx, \\
 (\bF(\Uhat_{j - 1}(t), \Uhat_j(t), \Uhat_{j + 1})(t))_{i\ell} =& 
  \hat f_{\ell}(\U_j(t, x_{j + 1/2}^-), \U_{j + 1}(t, x_{j + 1/2}^+))
  \varphi_i(1 / 2) \label{eq:numFlux} \\
 & - \hat f_{\ell}(\U_{j - 1}(t, x_{j - 1/2}^-),
  \U_j(t, x_{j - 1/2}^+))\varphi_i(-1 / 2),\\
 (\bV(\Uhat_j(t)))_{i\ell} =& \int_{I_j} f_{\ell}(\U_j(t, x)) \partial_x 
  \varphi_i \left( \frac{x - x_j}{\dx} \right)~dx,\\
 (\bS(\Uhat_j(t)))_{i\ell} =& \int_{I_j} s_{\ell}(\U_j(t, x))
 \varphi_i \left( \frac{x - x_j}{\dx} \right)~dx,
\end{align}
\end{subequations}
where $u_{\ell, t = 0}$, $\hat f_{\ell}$, $f_{\ell}$, and $s_{\ell}$ are the
$\ell$-th components 
of $\Uzero$, $\hat \bff$, $\bff$, and $\bs$ respectively.
Notice that $\bM$ is diagonalized by the choice of an orthogonal basis $\{ 
\varphi_i \}$.
We can write \eqref{eq:weakform_m1} as the ordinary differential equation
\begin{align}
\label{eq:ode1}
\partial_t \Uhat_j &= L_h(\Uhat_{j - 1}, \Uhat_j, \Uhat_{j + 1}), \quad 
\textrm{for } 
j \in \{1, \ldots , \ncells\} \text{ and } t \in (0, T),
\end{align}
with initial condition specified in \eqref{eq:weakform_m2}.

Boundary conditions are incorporated into the quantities $\U_0(t, x_{1/2})$
and $\U_{\ncells + 1}(t, x_{\ncells + 1/2})$, which we have not defined yet
but appear in the numerical flux \eqref{eq:numFlux} for the first and last
cells.
To define these terms, first we smoothly extend the definitions of
$\psiL(t, \mu)$  and $\psiR(t, \mu)$ to all $\mu$
(note that while moments are defined using integrals over all $\mu$, the
boundary conditions are in \eqref{eq:bcL}--\eqref{eq:bcR} only defined for
$\mu$ corresponding to incoming data),%
\footnote{
Although this is indeed the most commonly used approach, its inconsistency with 
the original boundary conditions \eqref{eq:bcL}--\eqref{eq:bcR}, is still an
open research topic 
\cite{Pomraning-1964,Larsen-Pomraning-1993-1, 
Larsen-Pomraning-1993-2,Struchtrup-2000,Levermore-2009}.%
}
then simply take $\U_0(t, x_{1/2}) := \vint{\basis \psiL(t, \cdot)}$ and
$\U_{\ncells + 1}(t, x_{\ncells + 1/2}) := \vint{\basis \psiR(t, \cdot)}$.
This completes the spatial discretization.

In this paper, except for some of the convergence tests in 
\secref{sec:manu-soln}, we use quadratic polynomials ($k = 2$) resulting in a 
third-order approximation.
The integrals in \eqref{eq:weakFormDefs} are computed using quadrature exact for 
polynomials of degree five to ensure the numerical scheme is third-order 
convergent.
We use the four-point Gauss-Lobatto quadrature rule since the function 
evaluations at the interval boundaries can be reused for the numerical fluxes 
$\bF$.

\subsection{Runge-Kutta time integration}
For a fully third-order method, we require a time-stepping
scheme for \eqref{eq:ode1} that is at least third-order.
We use the standard explicit SSP$(3,3)$ third-order strong
stability-preserving (SSP) Runge-Kutta  time discretization introduced in
\cite{CWShu}.
Let $\{t^\timelvl\}_{\timelvl=0}^{\ntime}$ denote time instants in $[0, \tf]$
with $t^\timelvl = \timelvl\Delta t$, and for each cell $j \in \{ 1, 2,
\ldots , \ncells \}$ let the initial coefficients $\Uhat_j(0)$ be defined as
in \eqref{eq:weakform_m2}.
Then for $\timelvl \in \{ 0, 1, \ldots , \ntime - 1 \}$ we compute
$\Uhat_j(t^{\timelvl})$ as follows:
\begin{align*}
 \Uhat_j^{(1)} &= \Uhat_j(t^{\timelvl}) + \dt 
  L_h(\Uhat_{j - 1}(t^{\timelvl}), \Uhat_j(t^{\timelvl}), \Uhat_{j + 1}
  t^{\timelvl})); \\
 \Uhat_j^{(2)} &= \dfrac{3}{4}\Uhat_j(t^{\timelvl}) + 
  \dfrac{1}{4}(\Uhat_j^{(1)} + \dt L_h(\Uhat_{j - 1}^{(1)}, \Uhat_j^{(1)},
  \Uhat_{j+1}^{(1)})); \\
 \Uhat_j(t^{\timelvl + 1}) &= \dfrac{1}{3}\Uhat_j(t^{\timelvl}) + 
  \dfrac{2}{3}(\Uhat_j^{(2)} + \dt L_h(\Uhat_{j - 1}^{(2)}, \Uhat_j^{(2)},
  \Uhat_{j + 1}^{(2)})).
\end{align*}
This specific Runge-Kutta method is a convex combination of forward Euler 
steps, a property which below helps us prove that the cell means of the 
internal stages are realizable.

A scheme of higher order could be achieved by increasing the degree $\order$ of 
the approximation space $V_h^k$ as well as the order of the Runge-Kutta 
integrator.
Unfortunately SSP-RK schemes with positive weights can at most be fourth order
\cite{Ruuth2004,Gottlieb2005}.
A popular solution is given by the so-called 
Hybrid Multistep-Runge-Kutta SSP methods.
A famous method is the seventh-order hybrid method in \cite{Huang2009} while 
recently two-step and general multi-step SSP-methods of high order have been
investigated \cite{Ketcheson2011,Bresten2013}.

\subsection{Numerical optimization}
\label{sec:Optimization}

In order to evaluate $\bff(\U)$ and $\br(\U)$ on the spatial quadrature points 
in each Runge-Kutta stage, we first compute the multipliers
$\alphahat(\U)$ solving the dual problem \eqref{eq:dual}.
For the Maxwell-Boltzmann entropy the dual objective function and its gradient 
are
$$
f(\bsalpha) := \Vint{\exp(\basis^T \bsalpha)} - \U^T \bsalpha \quand
\bg(\bsalpha) = \Vint{\basis \exp(\basis^T \bsalpha)} - \U,
$$
respectively.

We use the numerical optimization techniques proposed in \cite{ahot2013}.
The stopping criterion for the optimizer is given by
$$
\| \bg(\bsalpha) \|_2 < \tau,
$$
where $\| \cdot \|_2$ is the Euclidean norm, and $\tau$ is a user-specified 
tolerance, and we also use the isotropic regularization technique to return 
multipliers for nearby moments when the optimizer fails.%
\footnote{
The optimizer can fail for two reasons:  either the Cholesky 
factorization required to find the Newton direction fails or the number of 
iterations reaches a user-specified maximum $k_{\max}$.
}
Isotropically regularized moments are defined by the convex combination
$$
\bv(\bu, r) := (1 - r) \bu + r u_0 \uIso,
$$
where 
\begin{equation}
 \uIso = \frac12 \Vint{\basis}
\label{eq:uIso}
\end{equation}
is the moment vector of the isotropic density $\phi(\mu) \equiv 1/2$.
The form of $\bv(\bu, r)$ is also chosen so that $\bv(\bu, r)$ has the same 
zeroth-order moment as $\bu$.
We define for an outer loop an increasing sequence $\{ r_m \}$ for $m = 0, 1, 
\ldots, m_{\max}$.  We begin at $m = 0$ with $r = r_0 := 0$ and only increment 
$m$ if the optimizer fails to converge for $\bv(\bu, r_m)$ after $k_r$
iterations.
It is assumed that $r_{m_{\max}}$ is chosen large enough that the optimizer 
will always converge for $\bv(\bu, r_{m_{\max}})$ for any realizable $\bu$. 

\subsection{Realizability preservation and limiting}

In order to evaluate the flux-term $\bff(\U_j(t, x_{jq}))$ at the spatial quadrature nodes 
$x_{jq}$ in the $j$-th cell, we at least need $\U_j(t, x_{jq}) \in 
\RD{\basis}{}$ for each node, although when the angular integrals are 
approximated by quadrature, we in fact need $\U_j(t, x_{jq}) \in 
\RQ{\basis}$.
Unfortunately higher-order schemes typically cannot guarantee this, as
has been 
observed in the context of the compressible Euler equations (which are
indeed in the hierarchy of minimum-entropy models) in \cite{Zhang2010}.

We can, however, first show that, when the moments at the quadrature nodes are 
realizable, our DG scheme preserves realizability of the cell means 
$\Ubar_j(t)$ under a CFL-type condition.  
With realizable cell means available, we then apply a linear scaling limiter to 
each cell pushing $\U_j(t, x_{jq})$ towards the cell mean and thus into the 
realizable set for each node $x_{jq}$. 

Following the arguments in \cite{Zhang2010a,Zhang2011}, this limiter does not
destroy the accuracy of the scheme in case of smooth solutions if $\Ubar_j$
is not on the boundary of the realizable set.
We test this numerically in \secref{sec:limiter-tests}.

\subsubsection{Realizability preservation of the cell means}

To prove realizability preservation of the cell means we will need three main 
ingredients:  first, an exact quadrature to represent the cell means using 
point values from the cell; second, a representation of the moments collision 
operator; and finally a lemma that allows us to add the flux term without 
leaving the realizable set.

First, following \cite{Zhang2010,Zhang2011a}, we consider the 
$\nqs$-point Gauss-Lobatto rule and use its exactness for polynomials of degree 
$\order \le 2\nqs - 3$ to write the cell means as
\begin{equation}
 \Ubar_j(t) = \frac{1}{\dx} \int_{I_j} \U_j(t, 
  x) dx = \sum\limits_{q = 1}^{\nqs}w_q \U_{j}(t, x_{jq}),
\label{eq:cellMeanWithQuad}
\end{equation}
where $x_{jq} \in I_j$ are the quadrature nodes on the $j$-th cell, and $w_q$ 
are the weights for the quadrature on the reference cell $[-1/2, 1/2]$.
We choose the Gauss-Lobatto quadrature in particular because it includes the 
endpoints, that is
\begin{equation}
 x_{j - 1/2}^+ = x_{j1} \quand x_{j + 1/2}^- = x_{j\nqs}
\label{eq:GLquadEndPoints}
\end{equation}

Secondly, we note that since the collision kernel $T$ in
\eqref{eq:collisionOperator} is positive by assumption, the moments of the 
collision operator applied to the entropy ansatz can be written as the 
difference between a realizable moment vector $\UcollR(\U)$ and the given 
moment vector $\U$:
\begin{align}
\label{eq:CollisionOperatorMoments}
 \br(\U) = \Vint{\basis \collision{\ansatz_{\U}}} = \UcollR(\U) - \U.
\end{align}
For example, in the case of isotropic scattering ($T \equiv 1/2$), we have 
$\UcollR(\U) = u_0 \uIso$ (where $\uIso$ are the moments of the normalized
isotropic density, see \eqref{eq:uIso}).

Finally, we use the following lemma:
\begin{lemma} \label{lem:uPlusF}
If $\bu \in \RD{\basis}{}$ and $a - |b| \ge 0$, then $a\bu + b\bff(\bu) \in 
\RD{\basis}{}$.
\end{lemma}
\begin{proof}
Since $\bu \in \RD{\basis}{}$, then there exists a solution to the minimum 
entropy problem $\ansatz_{\bu}$, so that $\bu = \vint{\basis \ansatz_{\bu}}$ 
and $\bff(\bu) = \vint{\mu \basis \ansatz_{\bu}}$.  Thus
\begin{align*}
 a\bu + b\bff(\bu) &= \Vint{\basis \left( a + b\mu \right) \ansatz_{\bu}}.
\end{align*}
Since $\mu \in [-1, 1]$, the affine polynomial satisfies $a + b\mu \ge a - 
|b|$, so by assumption we have $(a + b\mu)\ansatz_{\bu} \ge 0$, which is a 
nonnegative measure representing $a\bu + b\bff(\bu)$.
\end{proof}

\begin{theorem}
\label{prop:eulerStepRealizable}
Assume $\source \ge 0$ and $2\nqs - 3 \ge k$, and let $\sig{t} := \sig{s} + 
\sig{a}$.
Consider the cell means at time instants $t^\timelvl$,
$$
\Ubar^\timelvl_j := \frac{1}{\dx} \int_{I_j} \U_j(t^\timelvl, x) dx.
$$
If the moment vectors $\U^\timelvl_{jq} := \U_j(t^\timelvl, x_{jq})$ at each spatial 
quadrature point $x_{jq}$ are in $\RD{\basis}{}$ (or $\RQ{\basis}$) and the 
CFL condition
\begin{align}
\label{eq:HO-CFL}
\frac{\dt}{\dx} < w_{\nqs}(1 - \sig{t}\dt).
\end{align}
holds, where $w_{\nqs}$ denotes the quadrature weight of the last quadrature 
weight of the $\nqs$-point Gauss-Lobatto quadrature rule on the reference cell 
$[-1/2, 1/2]$, then the cell means $\Ubar^{\timelvl + 1}_j$ computed by taking a 
forward-Euler time step for \eqref{eq:dweakform2} are also in $\RD{\basis}{}$ 
(respectively $\RQ{\basis}$).
\end{theorem}

\begin{proof}
The following arguments only use the fact that the realizable set 
is a convex cone, and therefore can be applied to either $\RD{\basis}{}$ or 
$\RQ{\basis}$ in exactly the same way.  For clarity of exposition, we also 
begin with the case $\source \equiv 0$.

By the orthogonality of our basis, $\Ubar_j(t) = \Uhat_j^{(0)}(t)$.  Therefore 
we use the subequations in \eqref{eq:weakform_m1} for $\Uhat_j^{(0)}$ (where, 
in particular we have $\partial_x \varphi_0 \equiv 0$).
We use the notation $\U_j^{\timelvl}(x) = \U_j(t^\timelvl, x)$, so that with a
forward-Euler approximation for the time derivative \eqref{eq:weakform_m1} gives
\begin{align*}
 \Ubar_j^{\timelvl + 1} = \Ubar_j^\timelvl + \dt &\left( 
 -\frac{\hat \bff(\U_j^\timelvl(x_{j + 1/2}^-),
  \U_{j + 1}^\timelvl(x_{j + 1/2}^+))
  - \hat \bff(\U_{j - 1}^\timelvl(x_{j - 1/2}^-), 
  \U_j^\timelvl(x_{j - 1/2}^+))}
  {\dx} \right. \\
  &\quad- \sig{a}\Ubar_j^\timelvl + \sig{s}\br(\U_j^\timelvl) \left.
  \vphantom{\frac12} 
\right).
\end{align*}
Now using the fact that the cell interfaces are the first and last 
quadrature nodes (recall \eqref{eq:GLquadEndPoints}), we now substitute the 
definition of $\hat \bff$ and the appropriate representation of the moments of 
the collision operator \eqref{eq:CollisionOperatorMoments}, then use the 
quadrature formula for the cell means \eqref{eq:cellMeanWithQuad}, and finally 
collect terms to get:
\begin{subequations}
\label{eq:uNplus1}
\begin{align}
\allowdisplaybreaks
 \Ubar_j^{\timelvl + 1} =& \Ubar_j^\timelvl + \dt \left( 
  -\frac{1}{2\dx} \left(
  \bff\left(\U^\timelvl_{j\nqs}\right) + \bff\left(\U^\timelvl_{(j + 1)1}\right)
  - \left(\U^\timelvl_{(j + 1)1} - \U^\timelvl_{j\nqs}\right) \right. \right.
  \nonumber \\
 &\qquad\qquad - \left. \left(\bff\left(\U^\timelvl_{(j - 1)\nqs}\right) + 
  \bff\left(\U^\timelvl_{j1}\right) - \left(\U^\timelvl_{j1}
  - \U^\timelvl_{(j - 1)\nqs}\right)\right) \right) \nonumber \\
 &\qquad\qquad - \sig{a}\Ubar_j^\timelvl + \sig{s}(\UcollR(\U_j^\timelvl)
  - \Ubar_j^\timelvl) \left. 
  \vphantom{\frac12} \right) \nonumber \\
 =& \sum\limits_{q = 1}^{\nqs}w_q \U^\timelvl_{jq} + \dt \left( 
  -\frac{1}{2\dx} \left(
  \bff\left(\U^\timelvl_{j\nqs}\right) + \bff\left(\U^\timelvl_{(j + 1)1}\right)
  - \left(\U^\timelvl_{(j + 1)1} - \U^\timelvl_{j\nqs}\right) \right. \right.
  \nonumber \\
 &\qquad\qquad\qquad\qquad - \left. \left(\bff\left(\U^\timelvl_{(j - 1)\nqs}
  \right) + 
  \bff\left(\U^\timelvl_{j1}\right) - \left(\U^\timelvl_{j1}
  - \U^\timelvl_{(j - 1)\nqs}\right)\right) \right) \nonumber \\
 &\qquad\qquad\qquad\qquad - \sig{t}\sum\limits_{q = 1}^{\nqs}w_q
  \U^\timelvl_{jq} + 
  \sig{s}\UcollR\left(\U_j^\timelvl\right) \left. \vphantom{\frac12} \right) 
  \nonumber \\
 =& \sum\limits_{q = 2}^{\nqs - 1} w_q (1 - \dt\sig{t}) \U^\timelvl_{jq} + \dt 
  \sig{s}\UcollR\left(\U_j^\timelvl\right) \label{eq:uNplus1-1} \\
 &+ \frac{\dt}{2\dx} \left( \U^\timelvl_{(j + 1)1}
  - \bff\left(\U^\timelvl_{(j + 1)1} \right) 
  + \U^\timelvl_{(j - 1)\nqs} + \bff\left(\U^\timelvl_{(j - 1)\nqs} \right)
  \right) \label{eq:uNplus1-2} \\
 &+ \left( w_1 - \frac{\dt}{2\dx} - \dt\sig{t}w_1 \right) \U^\timelvl_{j1} + 
  \frac{\dt}{2\dx}\bff\left( \U^\timelvl_{j1} \right) \label{eq:uNplus1-3} \\
 &+ \left( w_{\nqs} - \frac{\dt}{2\dx} - \dt\sig{t}w_{\nqs} \right) 
  \U^\timelvl_{j\nqs} - \frac{\dt}{2\dx}\bff\left( \U^\timelvl_{j\nqs} \right)
  \label{eq:uNplus1-4}
\end{align}
\end{subequations}
Keeping in mind that we have assumed that each moment vector $\U^\timelvl_{jq}$ 
is realizable, we consider each of the final lines:
\begin{itemize}
 \item If $\sig{t} \dt < 1$, a condition which is indeed weaker than 
  \eqref{eq:HO-CFL}, the expression in the first line, \eqref{eq:uNplus1-1}, is 
  a positive linear combination of realizable moments.  Since the realizable 
  set is a convex cone, this expression is realizable.
 \item The terms in the second line \eqref{eq:uNplus1-2} can be shown to be 
  realizable by two applications of \lemref{lem:uPlusF} with $a = 1$ and $b = 
  \pm 1$.
 \item The expressions in the last line, \eqref{eq:uNplus1-3} and 
  \eqref{eq:uNplus1-4}, are each realizable according to \lemref{lem:uPlusF} 
  and \eqref{eq:HO-CFL} (and recalling that $w_1 = w_{\nqs}$).
\end{itemize}
Finally, $\Ubar_j^{\timelvl + 1}$ is realizable since it is a sum of realizable 
moment vectors.

When $\source \ge 0$, notice that this simply adds to \eqref{eq:uNplus1}
the term $\vint{\basis \source}$, which is realizable and thus does not
affect the conclusion.
\end{proof}

Since we are using SSP-Runge-Kutta time-stepping schemes, whose stages are 
convex combinations of forward Euler steps, \thmref{prop:eulerStepRealizable} 
guarantees that under the appropriate CFL condition, the cell means for every 
Runge-Kutta stage are realizable.
In particular, the SSP(3, 3) which we use is a convex combination of Euler 
steps all with time step $\dt$.


\longpaper{
\begin{remark}
As in \cite{AllHau12} we need to modify the CFL condition a little bit more. 
Since we don't solve the optimization/closure problem analytically we make some 
numerical error, compare Section \ref{sec:NumQuad}.

For now, let $\psi_{j,\beta}$ be the (unknown) analytical solution of the 
optimization problem at quadrature point $x_\beta$ and $\hat\psi_{j,\beta}$ the 
approximative solution. Their moments are denoted accordingly by $\U_{j,\beta}$ 
and $\hat\U_{j,\beta}$, respectively. In the above proof we need to replace  
$\psi_{j,\beta}$ by the corresponding approximate solution $\hat\psi_{j,\beta} 
= \gamma_{j,\beta}\psi_{j,\beta}$ with $\gamma_{j,\beta} = 
\cfrac{\hat\psi_{j,\beta}}{\psi_{j,\beta}}$ an error indicator for the closure 
problem at quadrature point $\hat{x}_{j,\beta}$. Now, the last line of the 
proof has the form
\begin{align*}
\Xi_j^{(\timelvl)} 
&\hspace{-2pt}\geq\hspace{-5pt}\sum\limits_{\beta=2}^{\nqs-1}\hat{w}
_\beta\left(1-\dt\gamma_{j,\beta}\left(\sigma_{a,j,\beta}+\collconst_{j,\beta}
\right)\right) \psi_{j,\beta}^{(\timelvl)} + 
\left(\hat{w}_1\left(1-\dt\gamma_{1,\beta}\left(\sigma_{a,1,\beta}+\collconst_{1
,\beta}\right)\right)\right) \psi_{j-\frac12}^+
\\&\hskip 0.5cm -\gamma_{1,\beta}\frac{\dt}{\dx} \psi_{j-\frac12}^++ 
\left(\hat{w}_\nqs\left(1-\gamma_{\nqs,\beta}\dt\left(\sigma_{a,\nqs,\beta}
+\collconst_{\nqs,\beta}\right)\right)-\gamma_{\nqs,\beta}\frac{\dt}{\dx}
\right) \psi_{j+\frac12}^-
\end{align*}
Now the corresponding CFL condition reads
\begin{align}
\label{eq:HOCFL2}
\gamma\Delta t < \cfrac{\hat{w}_1\Delta 
x}{\left(\sigma_{a,\max}+\collconst_{\max}\right)\hat{w}_1\Delta x + 1}
\end{align}
where $\gamma = \max\limits_{j=1,\ldots \ncells}\{\gamma_{j\pm\frac12}^\mp\}$.
\end{remark}
}{
}

\subsubsection{Realizability-preserving limiter}
\thmref{prop:eulerStepRealizable} makes the assumption that the 
point-values of the local DG polynomials are realizable, and this must also
hold in order to evaluate the flux $\bff$ at the moment vectors on the
quadrature points.
This can be achieved 
by applying a linear scaling limiter in each cell.
Recall the definition of $\U_j(t, x)$:
\begin{align*}
 \U_j(t, x) = \sum\limits_{i=0}^{\order}\Uhat_j^{(i)}(t) \varphi_i 
  \left(\frac{x - x_j}{\dx}\right) = \bar{\U}_j(t) + \sum\limits_{i=1}^{\order} 
  \Uhat_j^{(i)}(t) \varphi_i \left( \frac{x - x_j}{\dx} \right).
\end{align*}
We can see that using convexity of the realizable set, if $\Ubar_j$ is 
realizable, then for each quadrature point there exists a $\theta \in 
[0,1]$ such that
\begin{align*}
 \U_j^\theta(t, x_{jq}) := \theta \Ubar_j(t) + (1 - \theta) \U_j(t, x_{jq})
\end{align*}
is realizable.  Indeed, by inserting the definition of $\U_j(t, x_{jq})$ from 
above, we can write the limited moment vector as
\begin{align*}
 \U_j^\theta(t, x_{jq})
  = \Ubar_j(t) + (1 - \theta) \sum\limits_{i = 1}^{\order} 
  \Uhat_{j}^{(i)}(t) \varphi_i\left( \frac{x_{jq} - x_j}{\dx} \right);
\end{align*}
thus when limiting is necessary, the higher-order coefficients
are damped while the cell mean remains unchanged.
An example of the limiting process is illustrated in 
\figref{fig:QuadLimiterEx1}, which considers the following polynomial 
representation of an M${}_1$ solution:
$$
\begin{pmatrix} u_0\\u_1 \end{pmatrix} 
 = \begin{pmatrix} 1 & 0.5 &-0.2\\0.8 & 0.2 & 0.6 \end{pmatrix}
 \begin{pmatrix} 1 \\ 2x \\  \frac12(12x^2-1) \end{pmatrix}
$$

After limiting with $\theta > 9/11$ the vector $(u_0^\theta, 
u_1^\theta)$ becomes realizable.

\begin{figure}
\centering
\subfloat[Before limiting; notice the nonrealizable moments at the end points.]
{
\includegraphics[width=0.45\textwidth]{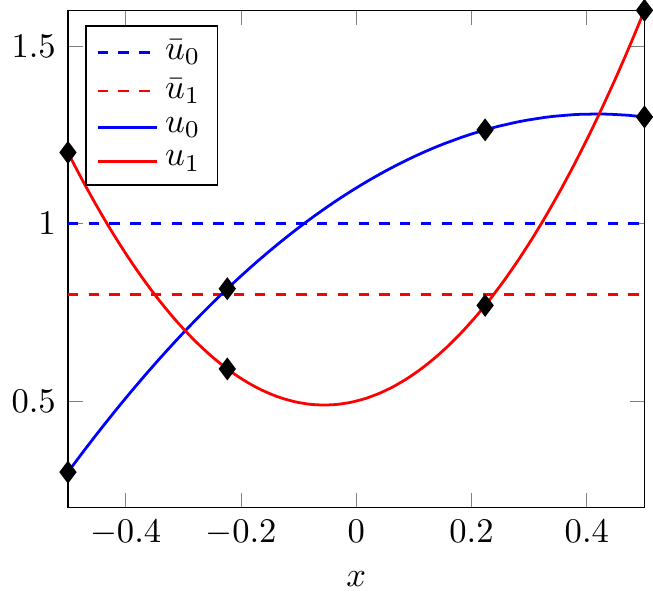}
}
\hspace{0.05\textwidth}
\subfloat[After limiting; realizable.]
{
\includegraphics[width=0.45\textwidth]{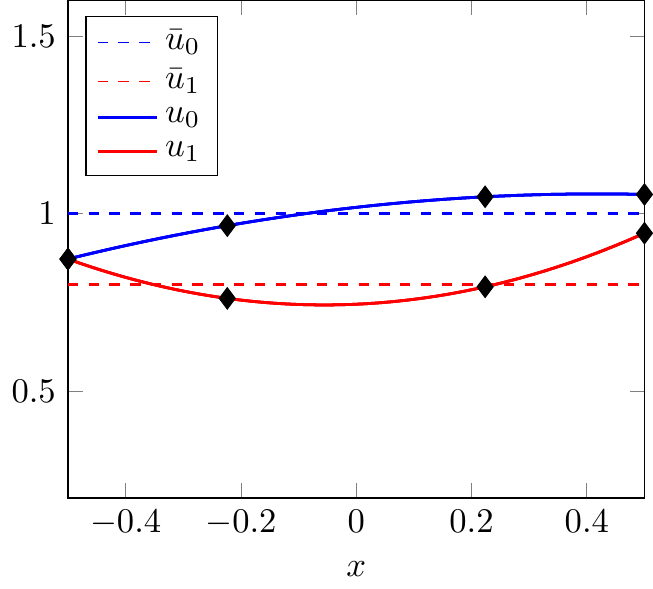}
}
\caption{Application of the realizability preserving limiter to the M${}_1$ 
system for a quadratic polynomial.  Here, $(u_0, u_1)$ are realizable only if 
$|u_1| < u_0$. Black squares indicate the spatial quadrature points $x_q$.}
\label{fig:QuadLimiterEx1}
\end{figure}

Since the full realizable set $\RD{\basis}{}$ is characterized by the 
positive-definiteness of the associated Hankel matrices, computing the smallest 
$\theta$ such that $\U_j^\theta(t, x_{jq}) \in \RD{\basis}{}$ is in general 
difficult.
Furthermore, for higher-dimensional problems (when dimension of the angular
domain is more than one) the realizable set is in general not 
well-understood.

This computation is however easier for the numerically realizable set 
$\RQ{\basis}$ using its \textit{half-space representation}.
This representation is also intriguing because it extends to 
higher-dimensional problems, since even in these cases $\RQ{\basis}$ remains 
the interior of a cone generated by a convex polytope.

In the following
for clarity of exposition we omit the time arguments and 
spatial-cell indices, therefore using $\Ubar$ to indicate the (always 
realizable) moment vector at the cell mean and $\U_q$ to indicate the (not 
necessarily realizable) moment vector at a quadrature point.
In an implementation, the limiter is applied to every quadrature point
in every cell.

To discuss the computation of $\theta$, first we assume that the moment
vectors $\Ubar$ and $\U_q$ have been scaled such that $\max(\bar{u}_0, 
u_{q0}) = 1/2$, where $\bar{u}_0$ and $u_{q0}$ are the zero-th components of
$\Ubar$ and $\U_q$, respectively.%
\footnote{
Here we are using $1/2$ instead of $1$ to ensure that $\Ubar$ and $\U_q$ are 
in the \textit{interior} of $\RQsone{\basis}$.
}
Then the limited moment vector $\U_q^\theta := \theta \Ubar + (1 - \theta) 
\U_q$ satisfies $u^\theta_{q0} < 1$ for all $\theta \in [0, 1]$, and
without loss of generality we can apply the limiter to move the moments into
the bounded set
$$
\RQsone{\basis} = \left\{\U\in \RQ{\basis}~|~ u_0 < 1\right\}.
$$
instead of the full, unbounded set $\RQ{\basis}$.
Now, using that $\RQ{\basis}$ is the cone generate by $\RQone{\basis}$
and Proposition \ref{prop:RQ}, $\RQsone{\basis}$ can be written as the
interior of a convex polytope
$$
\RQsone{\basis} =  \left\{ \lambda \bu : \lambda \in (0, 1) \text{ and }
 \bu \in \RQone{\basis} \right\}
 = \interior \co \left\{0, \basis(\mu_1), \basis(\mu_2), \ldots ,
 \basis(\mu_{\nqmu}) \right\}.
$$
As a convex polytope, it has a half-space representation
\cite{brondsted-convex-polytopes} of the form
$$
 \RQsone{\basis} = \left\{ \U \in \R^{\nmom + 1}:~ \ba_i^T\U < b_i, i \in \{1, 
\ldots, d\} \right\},
$$
where $d$ is the number of facets of the polytope.
For each $i$-th facet, the vectors $\ba_i$ and scalars $b_i$ are computed
from the set of vertices defining the facet.
These sets of vertices can be computed using standard convex-hull algorithms,
and for our implementation we used the Matlab routine
\texttt{convhulln}.
These are fixed throughout the computation and therefore can be precomputed.

Candidate values $\theta_{qi}$ for the $\theta$ which ensures $\U_q^\theta
\in \RQsone{\basis}$ can now be computed for each $i$-th facet by
$$
 \ba_i^T(\theta_{qi} \Ubar + (1 - \theta_{qi}) \U_q) = b_i
 \quad \iff \quad
 \theta_{qi} = \frac{b_i - \ba_i^T\U_q}{\ba_i^T(\Ubar - \U_q)}.
$$
If we wanted to limit exactly to $\partial \RQ{\basis}$, we would first take
$$
\theta_q^{\partial \cR} := \begin{cases}
             0 & \text{if there is no } \theta_{qi} \in [0, 1], \\
             \max\{\theta_{qi} : \theta_{qi} \in [0, 1] \} & \text{else;}
            \end{cases}
$$
and then for the cell in question, we would apply the largest $\theta_q$ from 
the quadrature nodes:
$\theta^{\partial \cR} := \max\{\theta_q^{\partial \cR}\}$.
This would ensure that the moment vectors at each quadrature node in the cell 
are in the realizable set or on its boundary.

In practice, however, we do not want to choose $\theta_q$ such that the limited 
moment vector $\theta_q \Ubar + (1 - \theta_q) \U_q$ lies exactly on the 
boundary of the realizable set $\partial \RQ{\basis}$, but rather so that 
the limited moment vector is in the interior of $\RQ{\basis}$.
Therefore we define a tolerance $\eps$ to add to each relevant $\theta_{qi}$ 
(that is in $[0, 1]$) as well as those facets such that $\theta_{qi} \in [-\eps, 
0]$, indicating that while $\U_q$ is on the correct side of the half space, it 
is closer than $\eps$ to the facet.
Keeping in mind that $\theta_q$ should not exceed one, this gives
$$
\theta_q := \begin{cases}
             0 & \text{if there is no } \theta_{qi} \in [-\eps, 1], \\
             \min \{ 1, \eps + \max\{\theta_{qi} : \theta_{qi} \in [-\eps, 1] \} 
              \} & \text{else.}
            \end{cases}
$$
Finally, in the implemented version, for the cell we set $\theta = \max \{ 
\theta_q \}$.
Figure \ref{fig:QuadLimiter} shows an example for the full moment model with 
$\nmom=2$.

\InsertFig{
\subfloat[Limiter example for $\nmom=2$ full moments $\mbasis$-basis with
$\nqmu = 7$ and, for simplicity of exposition, $u_{q0} = 
\bar{u}_0 = 1$.\label{fig:QuadLimiter}]
{
\includegraphics[width=0.45\textwidth]{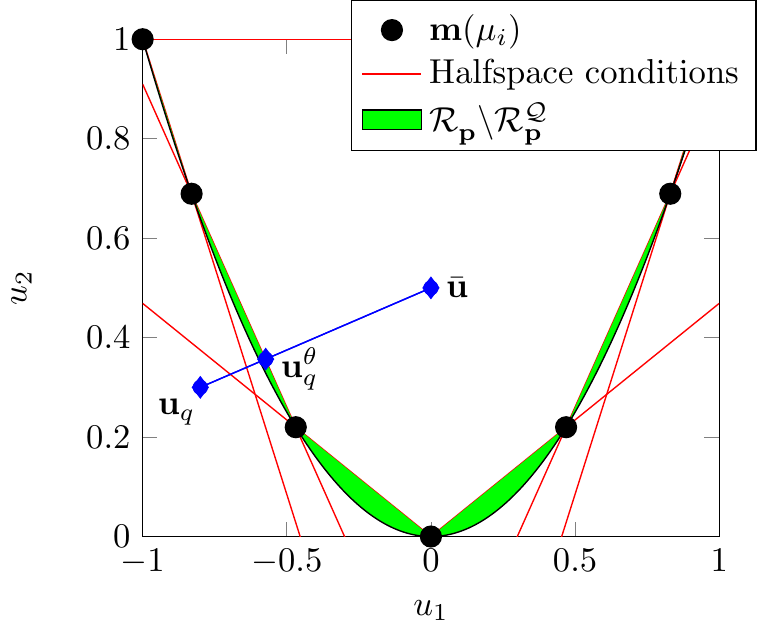}
}%
\hspace{0.05\textwidth}%
\subfloat[The number of facets of $\RQ{\basis}$ for a few $\MN$ (for
$N \in \{4, 6, 8\}$) and $\MMN$ (for $N \in \{2, 3, 4\}$) models.
The number of facets of MM$_2$ and MM$_3$ are almost exactly the
same.
\label{fig:Facets}]
{
\includegraphics[width=0.45\textwidth]{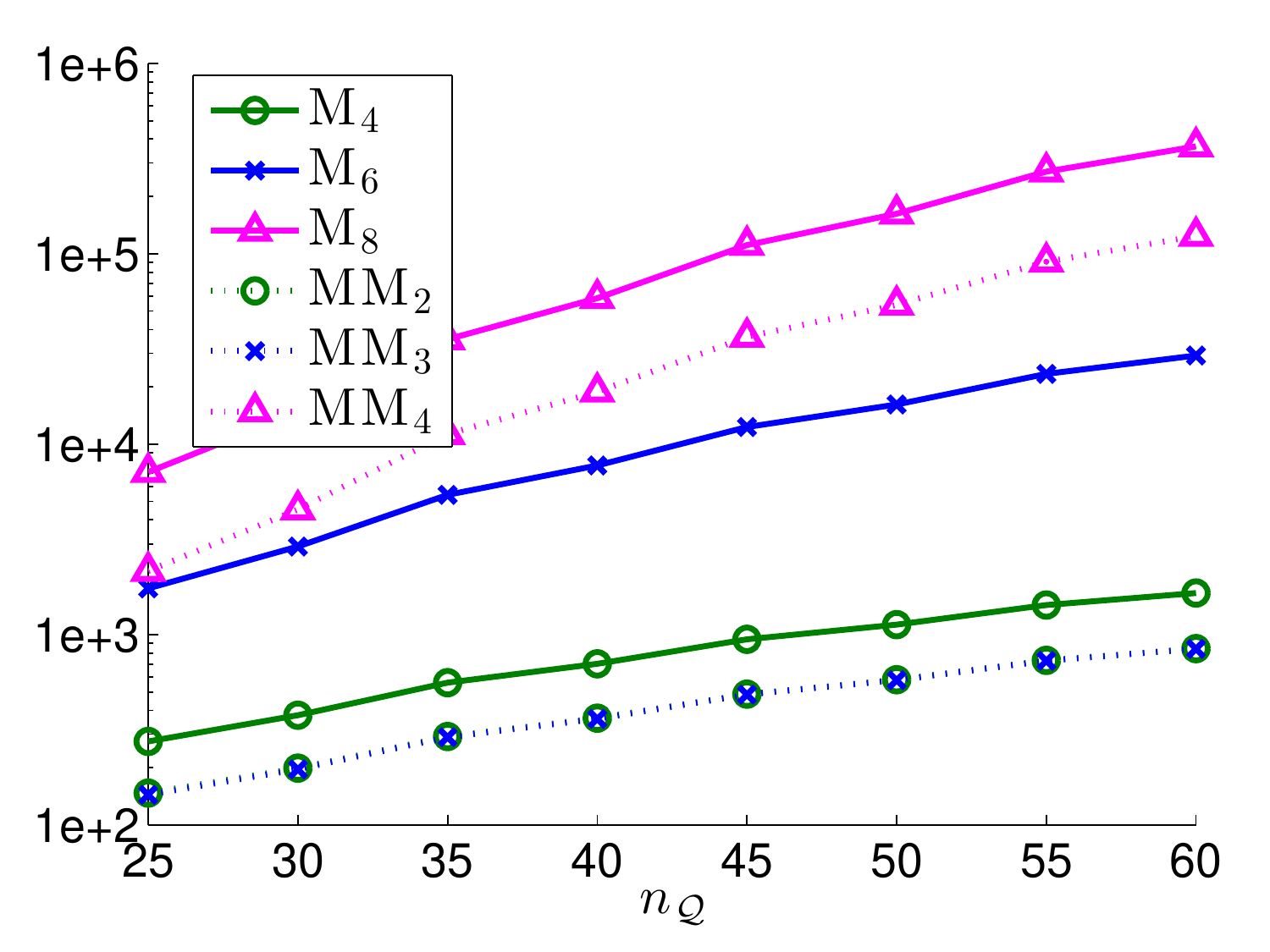}
}

}{}{}

The main drawback to this implementation of the limiter is that, as
illustrated 
in \figref{fig:Facets} the number of facets $d$ grows rapidly both with 
the number of moments $\nmom$ and the number of quadrature points $\nqs$.
The numbers for this figure were computed using results from the study
of convex polytopes, and
a more detailed discussion is in the appendix.
This issue was not a significant obstacle for our simulations but will be in higher dimensions where a much larger number of quadrature 
points is necessary.
One clear speed up is that not every $\theta_{qi}$ should be computed, as many 
facets do not intersect the line between $\Ubar$ and $\U_q$.
More importantly, however, an implementation in higher-dimensions will probably 
have to further approximate $\RQ{\basis}$ by removing some irrelevant facets.

\begin{remark}
\label{rem:NumberOfQuadraturePoints}
There is actually no loss in accuracy by limiting to $\RQ{\basis}$
instead of $\RD{\basis}{}$ when the quadrature $\cQ$ defining $\RQ{\basis}$
is the same as the quadrature used to compute all angular integrals:
In this case we are solving the moment equations
\eqref{eq:moment-closure} with quadrature approximations to the angular
integrals, and the numerical solution indeed lives on $\RQ{\basis}$.
(Furthermore, the function $\bff$ can only be evaluated on $\RQ{\basis}$.)
The numerical solutions using quadrature then converge to moment equations
with exact integrals as the quadrature rule converges.
\end{remark}


\longpaper{
\paragraph{Vertex-based representation}
\comment{Complexity is $\mathcal{O}(n^{3.5})$ if $n$ is the dimension of the
solution vector, using Karmarkas algorithm \cite{hamacher2006linear}}
\comment{OK sounds good, but I'll check with my old advisor anyway.}
For $\nmom$ reasonably big it may be more appropriate to solve the limiter 
problem in the vertex-based representation of the polytope using the original 
definition of realizability \eqref{eq:QRealizable}:
\begin{subequations}
\label{eq:LP}
	\begin{align}
\min\quad&\theta&&&	\min \quad&c^T\nu\\
	s.t.\quad &\sum\limits_{i=1}^{\nqmu}\lambda_i\basis(\mu_i) 
=(1-\theta)\bu+\theta\bar{\bu}&\Longleftrightarrow&&s.t. \quad &A\nu = \bu\\
	&\theta,\lambda_i\geq 0&&&&\nu\geq 0
	\end{align}
	\end{subequations}
	with $\nu = \begin{pmatrix}
	\lambda\\\theta
	\end{pmatrix}$, $c = (0,\ldots,0,1)^T$ and $A = \begin{pmatrix}
	M&\bu-\bar{\bu}
	\end{pmatrix}$ with $M_{ij} = \basis_i(\mu_j)$.
Dualizing this linear program (to avoid the equality constraints) gives 
\cite{hamacher2006linear}

\begin{subequations}
\label{eq:DLP}
\begin{align}
	\max\quad&\omega^T\bu\\
	s.t.\quad &\omega^TA\leq c
	\end{align}
	\end{subequations}

Using the strong duality theorem \cite{hamacher2006linear} we get that for an 
optimal solution of \eqref{eq:DLP} and \eqref{eq:LP} satisfy $\theta = c^T\nu = 
\omega^T\bu$

Therefore solving the dual LP gives immediately the optimal value for $\theta$. 
Note that it is again helpful to choose normalized moments for numerical 
stabilization of the problem. 
\longpaper{
This results in the following algorithm:

\begin{algorithm}[H]
  \SetKwData{Left}{left}\SetKwData{This}{this}\SetKwData{Up}{up}
  \SetKwFunction{Union}{Union}\SetKwFunction{FindCompress}{FindCompress}
  \SetKwInOut{Input}{Input}\SetKwInOut{Output}{Output}
  \Input{$\U = \left(u_0,\ldots,u_{\Nmom-1}\right)\in\R^{\Nmom}$, $\bar{\U} = 
\left(\bar{u}_0,\ldots,\bar{u}_{\Nmom}\right)\in\R^{\Nmom}$ realizable, safety 
factor $\eps$}
  \Output{$\hat{\U}$ realizable}
  \BlankLine
  Set $\phi = \cfrac{\U}{\max(u_0,\bar u_0)}$, $\bar{\phi} = 
\cfrac{\bar{\U}}{\max(u_0,\bar u_0)}$\;
  Set $c = (0,\ldots,0,1)^T$ and $A = \begin{pmatrix}
  	M&\phi-\bar{\phi}
  	\end{pmatrix}$ with $M_{ij} = \basis_i(\mu_j)$\;
  Solve \eqref{eq:DLP} for $y$\;
 \If{$\theta>0$}{
   $\theta = \min(\theta+\eps,1)$\;}
  Set $\hat{\U} = \left(\theta \bar{\U} + (1-\theta)\U\right)$\;
  \caption{DLP-Limiter for quadrature based algorithms.}\label{alg:DLPLimiter}
\end{algorithm} 
}{}

Here we only have a polynomial number of variables in this linear program but 
we can't precompute the solution efficiently. However, for large $\nmom$ this 
is still more efficient than the facet-based limiter.
\comment{That is true. But it is clear that there is an $N$ where this
intersects. I bet that this is indeed relevant in higher dimension.}
\comment{Sure, but if we're going to include this section, I still think we 
should include some kind of experiments or discussion about whether this 
intersection point is reasonably close.  And anyway, I don't think this will be 
hard to do.}
\comment{I agree.}
}{}

\subsection{Slope limiting}
Although the realizability limiter plays some role in dampening spurious
oscillations in numerical solutions, further dampening is needed
\cite{Olbrant2012}.
We use the standard TVBM corrected minmod limiter proposed in
\cite{CockburnShuPk}.

Assuming that the major part of the spurious oscillations are 
generated in the linear part of the underlying polynomial, whose slope in
the $j$-th cell is 
simply $\Uhat_j^1$, a basic limiter can be defined as
\begin{align*}
 \Lambda^{\text{scalar}}_j(\Uhat) = \left\{ 
  \begin{array}{cc}
   \left( \begin{array}{c}
           \left( \Uhat^0_j \right)^T \\
           \left( D_j(\Uhat) \right)^T \\
           (0, 0, \ldots , 0)
          \end{array} \right)
    &\text{ if }
    \begin{array}{c}
     \abs{\Uhat_j^1} \geq M(\dx_j)^2 \text{ and,} \\
     D_j(\Uhat) \ne \Uhat_j^1
    \end{array} \\
    \Uhat_j & \text{otherwise}
  \end{array} \right.
\end{align*}
for the $j$-th cell and the case $\order = 2$, that is piecewise quadratic
approximations, so that the final row of zeros in the first case indicates
that the coefficients for the quadratic basis functions are set to zero for
each moment component.
The absolute value and inequality are applied component-wise, and
\begin{align*}
 D_j(\Uhat) = D(\Uhat_{j - 1}, \Uhat_j, \Uhat_{j + 1}) := 
  m(\Uhat_j^1, \Uhat_{j + 1}^0 - \Uhat_j^0, 
  \Uhat_j^0 - \Uhat_{j - 1}^0).
\end{align*}
The label ``scalar'' is used because the limiter is directly applied to each 
scalar component of $\U_h$.
The function $m$ is the standard minmod function applied component-wise:
\begin{align*}
m(a_1,a_2,a_3) &= \begin{cases}
 \operatorname{sign}(a_1) \min\{|a_1|,|a_2|,|a_3|\} & \text{if } 
 \operatorname{sign}(a_1) = \operatorname{sign}(a_2) = 
 \operatorname{sign}(a_3), \\
0 & \text{else}.
\end{cases}
\end{align*}

The constant $M$ is a problem-dependent estimate of the second derivative, 
though we note that in \cite{CockburnShuPk} the authors did not find the 
solutions very sensitive to the value chosen for this parameter.

However, it has been found that applying the limiter to the components 
themselves may introduce non-physical oscillations around an otherwise monotonic 
solution \cite{Cockburn1989}.
Therefore we instead apply the limiter to the local characteristic fields of the 
solution.
The characteristic fields are found by transforming the moment vector $\U$ 
using the matrix $\bV_j$, whose columns hold the eigenvectors of the Jacobian 
$\partial \bff / \partial \U$ evaluated at the cell mean $\Ubar_j$.
We then transform back to the moment variables after applying the limiter.
In the end, since $\Uhat_j$ is a matrix of size $(\order + 1) \times (\nmom + 
1)$, this transformation is accomplished by post-multiplying with $\bV_j^{-T}$ 
so that
$$
 \Lambda_j(\Uhat) = \Lambda^{\text{scalar}}_j(\Uhat \bV_j^{-T}) 
\bV_j^T,
$$
where $\Uhat \bV^{-T}_j$ is understood as $(\ldots , \Uhat_{j - 1} \bV^{-T}_j, 
\Uhat_j \bV^{-T}_j, \Uhat_{j + 1} \bV^{-T}_j, \ldots)$.
We apply this limiter to every Runge-Kutta stage.

The Jacobian is computed at the cell means $\Ubar_j$ using 
\eqref{eq:fluxJacobian}.
This indeed also implies that we must solve the dual problem \eqref{eq:dual} to 
compute $\alphahat(\Ubar_j)$ for each cell.
For cases when the matrix $\bV_j$ has a condition number (defined using the
two-norm) greater than a tolerance $\kappa_{\text{jac}}$, we apply the
scalar limiter.

There are more advanced limiter strategies, for example WENO limiting 
\cite{Qiu2005,Zhao2013} or generalized minmod-limiting 
\cite{Krivodonova2007}, removing the drawback of having a 
problem-dependent parameter $M$.
The slope limiter is however not a focus of this work.
\section{Numerical results}
\label{sec:SIM}
We used the following parameter values:

\begin{tabular}{r @{$\:$} c @{$\:$} l l}
  $\tau$ &=& $10^{-9}$ & Optimization gradient tolerance, \\
  $\{ r_m \}$ &=& $\{ 0, 10^{-8}, 10^{-6}, 10^{-4}, 10^{-3} \}$
   & Outer regularization loop in optimizer, \\
  $k_r$ &=& $50$ & Number of optimization iterations before \\
   &&& advancing outer regularization loop, \\
  $\nqmu$ &=& $40$ & Number of angular quadrature points, \\
  $M$ &=& $50$ & Slope-limiter constant, value suggested \\
   &&& in \cite{CockburnShuPk}, \\
  $\eps$ &=& $10^{-14}$ & Realizability limiter tolerance, \\
  $\kappa_{\text{jac}}$ &=& $10^{5}$ & Condition-number tolerance for applying
   the characteristic \\
   &&& transformation in the slope limiter.
\end{tabular}

For the angular quadrature we used $(\nqmu / 2)$-point Gauss-Lobatto rules
over both $\mu \in [-1, 0]$ and $\mu \in [0, 1]$.
At the first time step, the initial multipliers for the optimizer
(recall \secref{sec:Optimization}) are those of the isotropic density with
the appropriate zeroth-order moment.
For the rest of the simulation, the initial multipliers are those from the
same point in space at the previous time step.

To set the time step we use condition \eqref{eq:HO-CFL} with equality.

\subsection{Convergence tests}

We numerically test the convergence of the scheme two ways:
First, we use the method of manufactured solutions on the total scheme, and
second we focus on the convergence of the spatial reconstructions using the
realizability limiter near the boundary of realizability.

\subsubsection{Manufactured solution}
\label{sec:manu-soln}
In general, analytical solutions for minimum-entropy models are not known.
Therefore, to test the convergence and efficiency of our scheme, we use the
method of manufactured solutions.
To avoid the effects of the boundary we use periodic boundary conditions,
and we set the spatial domain to $X = (-\pi, \pi)$.

We begin by defining a kinetic density in the form of the entropy ansatz
and which is periodic in space for every $t$:
\begin{subequations}
\label{eq:MFSM3}
\begin{align}
 \phi(t, x, \mu) =& \exp(\alpha_0(t, x) + \alpha_1(t, x) \mu ), \\
 \alpha_0(t, x) =& -K - \sin(x - t) + c_0 t - c_1,\\
 \alpha_1(t, x) =&  K + \sin(x - t).
\end{align}
\end{subequations}
A source term is defined by applying the transport operator to $\phi$:
$$
\source(t, x, \mu) := \partial_t \phi(t, x, \mu) + \mu \partial_x
\phi(t, x, \mu).
$$
Thus by inserting this $\source$ into \eqref{eq:FokkerPlanck1D} (and taking
$\sig{a} = \sig{s} = 0$)
we have that $\phi$ is, by construction, a solution of
\eqref{eq:FokkerPlanck1D}.

A tedious but straightforward computation shows that choosing $c_0 = 4$ gives
$\source \ge 0$, which means that \thmref{prop:eulerStepRealizable} will
apply to the resulting moment system (for any $K$).
Furthermore we take
\begin{align*}
 c_1 &= c_0 \tf - K + 1 - \log\left( \cfrac{K - 1}{2\sinh(K - 1)} \right)
\end{align*}
so that the maximum value of $\vint{\phi}$ for $(t, x) \in [0, \tf] \times X$
is one.
The parameter $K$ can be increased to make $\phi$ look increasingly
like a Dirac delta at $\mu = 1$.

Since our solution has the form of an entropy ansatz,
$\bv = \vint{\basis \phi}$ is also a solution of \eqref{eq:moment-closure}
whenever $1$ and $\mu$ are in the linear span of the basis functions
$\basis$.
Clearly this holds for the M$_{\nmom}$ and MM$_{\nmom}$ models for
$\nmom \ge 1$.
Notice also that $\bv$ approaches the boundary of realizability as $K$ is
increased.

We used the final time $\tf = \pi / 5$ and
chose $K = 55$, for which the maximum value of $u_1 / u_0$
is about $0.98$ (recall that $|u_1 / u_0| < 1$ is necessary for
realizability).
In the following, we used the M$_3$ model so that our results included the
effects of the numerical optimization.

We compute errors in the zero-th moment of the solution, which we denote
$v_0(t, x) = \vint{\phi(t, x, \cdot)}$.
Then $L^1$ and $L^\infty$ errors for the zero-th moment $u_{0,h}(t, x)$
(that is, the zero-th component of a numerical solution $\U_h$) are
defined as
\begin{equation}
 E^1_h = \int_X \left| v_0(\tf, x) - u_{0,h}(\tf, x) \right| dx
  \quand
 E^\infty_h = \max_{x \in X} \left| v_0(\tf, x) - u_{0,h}(\tf, x) \right|,
\label{eq:errors}
\end{equation}
respectively.
We approximate the integral in $E^1_h$ using a 100-point Gauss-Lobatto
quadrature rule over each spatial cell $I_j$, and $E^\infty_h$ is
approximated by taking the maximum over these quadrature nodes.
The observed convergence order $\nu$ is defined by
\begin{equation}
 \frac{E^p_{h1}}{E^p_{h2}} = \left( \frac{\dx_1}{\dx_2} \right)^\nu
\label{eq:conv-order}
\end{equation}
where for $i \in \{1, 2\}$, $E^p_{hi}$ is the error $E^p_h$ for the
numerical solution using cell size $\dx_i$, for $p \in \{1, \infty\}$.

A convergence table is presented in \tabref{tab:MFSM3} using a tight
gradient tolerance in the optimization of $\tau = 10^{-11}$.
We observe that the expected convergence rates are achieved both in $L^1$-
and $L^\infty$-errors, although for $\order = 0$, the solution has only
just begun to reach the convergent regime.
In \figref{fig:EfficiencyMFSM3} we plot the $L^\infty$-error versus the
computation time for the solution for the same value of $\tau$.
Here we clearly see that higher-order methods are more efficient.


The scheme is not convergent for arbitrarily large values of $K$.
For large values of $K$, the numerical solution will veer so close to the
boundary of the realizable set that the optimization will have to use
regularization, thus introducing errors into the solution.
This was observed in \cite{ahot2013}, though here we can display this
effect more precisely.
In \figref{fig:reg-kills}, we show the results using $K = 110$ for three
spatial discretizations.
In the most coarse discretization ($\ncells = 40$), regularization is never
necessary, but after doubling the number of cells, a few regularizations
are used and their effects can be seen in the figure.
Here, the optimizer regularizes four problems with $r = 10^{-8}$, around
$x = -1.76,\,-1.68,\,0.51,\,0.74$ at $t = 0.38,\,0.44,\,0.35,\,0.52$, respectively, and the effect
on the error spreads, mostly (to the right, the propagation direction of the solution), and magnifies slightly until the final time.
The observed convergence order for $E^1_h$ from $\ncells = 40$ to
$80$ is $\nu = 2.9$ while the corresponding $E_h^\infty$ order is $\nu = 2.1$.
When doubling the number of cells again (up to $\ncells = 160$), the
optimizer must regularize now only three problems with $r = 10^{-8}$, this time around
$x = -1.74,\,-1.66,\,1.04$ at $t = 0.56,\,0.51,\,0.44$, respectively.
Indeed, immediately from the figure one can see that convergence in the
$E^\infty_h$ error has stopped and the observed convergence rate in
$E^1_h$ from $\ncells = 80$ to $160$ is only $\nu = 1.6$.

\begin{table}
\centering
\begin{tabular}{r r@{.}l c r@{.}l c r@{.}l c}

 & \multicolumn{3}{c}{$\order = 0$} & \multicolumn{3}{c}{$\order = 1$} &
   \multicolumn{3}{c}{$\order = 2$}\\
\cmidrule(r){2-4} \cmidrule(r){5-7} \cmidrule(r){8-10}
$\ncells$ & \multicolumn{2}{c}{$E^1_h$} & $\nu$
 & \multicolumn{2}{c}{$E^1_h$} & $\nu$
 & \multicolumn{2}{c}{$E^1_h$} & $\nu$ \\ \midrule

 20&4&087e+00&---&3&524e-02&---&1&897e-05&---\\
 40&2&608e+00&0.6&9&532e-03&1.9&2&416e-06&3.0\\
 80&1&507e+00&0.8&2&482e-03&1.9&3&049e-07&3.0\\
160&8&161e-01&0.9&6&333e-04&2.0&3&828e-08&3.0\\
320&4&256e-01&0.9&1&600e-04&2.0&4&796e-09&3.0\\
640&2&174e-01&1.0&4&020e-05&2.0&6&072e-10&3.0\\ \midrule

 & \multicolumn{2}{c}{} & & \multicolumn{2}{c}{} & &
   \multicolumn{2}{c}{} & \\

 & \multicolumn{2}{c}{$E^\infty_h$} & $\nu$
 & \multicolumn{2}{c}{$E^\infty_h$} & $\nu$
 & \multicolumn{2}{c}{$E^\infty_h$} & $\nu$ \\ \midrule

 20&3&339e-01&---&3&164e-03&---&9&283e-06&---\\
 40&2&129e-01&0.6&8&543e-04&1.9&1&241e-06&2.9\\
 80&1&230e-01&0.8&2&222e-04&1.9&1&609e-07&2.9\\
160&6&662e-02&0.9&5&666e-05&2.0&2&048e-08&3.0\\
320&3&474e-02&0.9&1&431e-05&2.0&2&589e-09&3.0\\
640&1&775e-02&1.0&3&595e-06&2.0&3&365e-10&2.9\\
\end{tabular}
\caption{$L^1$- and $L^\infty$-errors and observed convergence order $\nu$
for the manufactured solution \eqref{eq:MFSM3} with optimization gradient
tolerance $\tau = 10^{-11}$.}
\label{tab:MFSM3}
\end{table}


\begin{figure}
\centering
\subfloat[$L^\infty$-Efficiency, $\tau = 10^{-11}$.  Here we used the same
   numbers of cells as in \tabref{tab:MFSM3}.]
{
\includegraphics[width=0.45\textwidth]{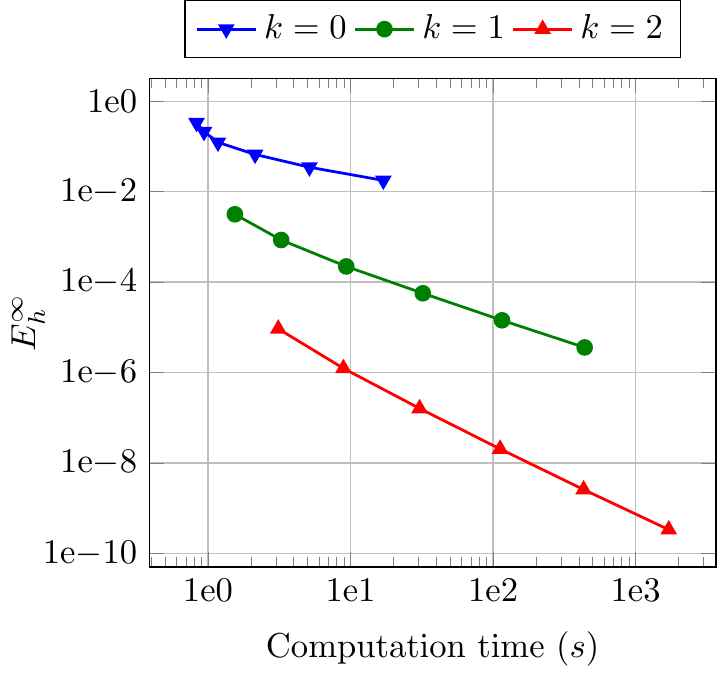}
\label{fig:EfficiencyMFSM3}
}
\hspace{0.05\textwidth}
\subfloat[Regularization destroying convergence for $K=110$.]
{\includegraphics[width=0.45\textwidth]{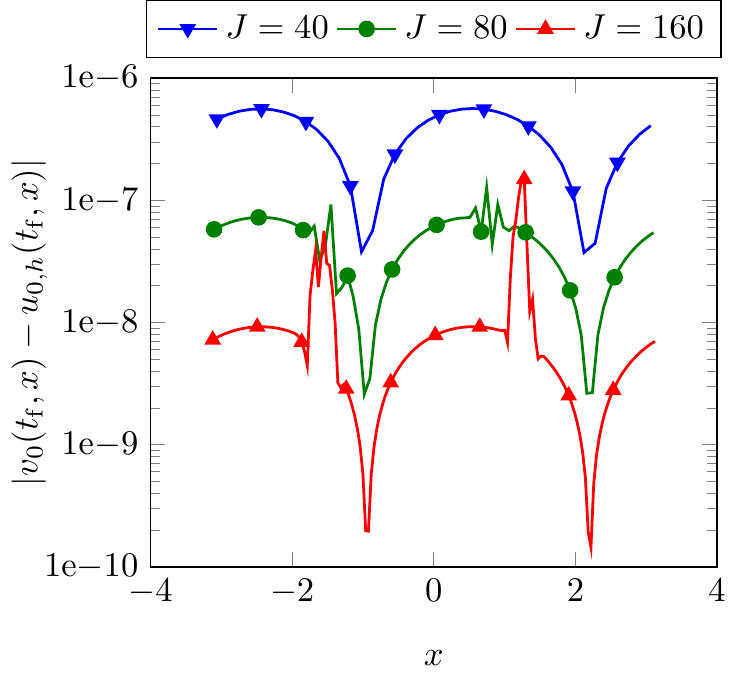}
\label{fig:reg-kills}}
\caption{Using the manufactured solution to consider the efficiency of
  higher-order methods and the effects of regularization.}
\label{fig:QuadLimiterEx1}
\end{figure}

\subsubsection{Convergence of the limited spatial reconstructions}
\label{sec:limiter-tests}

Despite choosing a manufactured solution which lies close to the boundary of
realizability, the realizability limiter was never active in any of our
simulations in \secref{sec:manu-soln}.
Therefore in this section we artificially define a curve of moment vectors
in space, reconstruct this curve in the finite-element space $V_h^2$ of
discontinuous quadratic polynomials (recall \eqref{eq:dg-space}), and apply
the limiter to move the reconstruction back into the set of numerically
realizable moments.
We then measure the convergence of this limited reconstruction.%
\footnote{
We would like to thank an anonymous reviewer for suggesting such a test.
}

For simplicity, we only consider the monomial basis $\basis = \mbasis$.
Using the Dirac delta function $\delta = \delta(\mu)$, we first choose two
moment vectors $\U_0$ and $\U_1$ which lie arbitrarily close to
the boundary of the numerically realizable set:
\begin{align*}
 \U_0 &:= (1 - \gamma) \Vint{\mbasis \delta(\mu - \mu_0)} + \gamma \uIso
  = (1 - \gamma) \mbasis(\mu_0) + \gamma \uIso \\
 \U_1 &:= 10^{-8} \left( (1 - \gamma) \Vint{\mbasis \delta(\mu + 1)}
  + \gamma \uIso \right)
  = 10^{-8} \left( (1 - \gamma) \mbasis(-1) + \gamma \uIso \right)
\end{align*}
where $\mu_0 \approx 0.5403$, which is a node from the angular quadrature
we use (see the beginning of \secref{sec:SIM}),
$\uIso$ are the moments of the isotropic density (see \eqref{eq:uIso}),
and $\gamma \in [0, 1]$ controls the distance to the boundary.
For $N > 1$, both $\U_0$ and $\U_1$ lie on the boundary of the realizable set
when $\gamma = 0$.  Since the Dirac delta functions are placed at
nodes in the angular quadrature, $\U_0$ and $\U_1$ (and any convex
combination thereof) are in $\RQ{\mbasis}$ for $\gamma \in (0, 1]$, and
so we define a curve of moments in space by taking convex
combinations of $\U_0$ and $\U_1$:
\begin{equation}
 \U(x) := (1 - \lambda(x)) \U_0 + \lambda(x) \U_1, \quad x \in [-1, 1],
\label{eq:u-limit-test}
\end{equation}
where $\lambda(x) \in [0, 1]$ is given by
$$
\lambda(x) := \frac{\cos(\pi x) + 1}{2}, \quad x \in [-1, 1].
$$

To perform the convergence test, we project $\U(x)$ onto $V_h^2$ for
increasing numbers of cells $\ncells$ and then apply the realizability
limiter.
Errors and observed convergence order are computed as in \eqref{eq:errors} and
\eqref{eq:conv-order}, respectively, over the finest grid of the tests, here
$\ncells = 512$ cells.

We found that taking $\gamma \in [10^{-10}, 10^{-2}]$ places the moment curve
$\U(x)$ close enough to the boundary of realizability that the realizability
limiter was active for every number of cells we considered.
In \tabref{tab:limiter-test-10} we show convergence rates for the smallest
$\gamma$ in this range.
These results show the desired third-order convergence.
In this table we include the column $\theta_{\max}$, which gives the
maximum value of $\theta$ from the realizability limiter over all spatial
cells.
That $\theta_{\max}$ is nonzero in each row indicates that the
realizability limiter was active for every reconstruction.
We observed similar results for every moment component and moment curves of
any order $N \in \{2, 3, 4, 5\}$.

\begin{table}
\centering
\begin{tabular}{r r@{.}l c r@{.}l c r@{.}l c r@{.}l c r@{.}l}

 & \multicolumn{6}{c}{$u_0$} & \multicolumn{6}{c}{$u_1$}
 & \multicolumn{2}{c}{ } \\
\cmidrule(r){2-7} \cmidrule(r){8-13}
$\ncells$ & \multicolumn{2}{c}{$E^1_h$} & $\nu$
 & \multicolumn{2}{c}{$E^\infty_h$} & $\nu$
 & \multicolumn{2}{c}{$E^1_h$} & $\nu$
 & \multicolumn{2}{c}{$E^\infty_h$} & $\nu$
 & \multicolumn{2}{c}{$\theta_{\max}$} \\ \midrule
 
  8 & 1&072e-03 & ---  & 1&848e-03 & ---  & 5&790e-04 & ---  & 9&983e-04 & ---  & 1&56e-02 \\
 16 & 1&130e-04 & 3.2 & 2&469e-04 & 2.9 & 6&108e-05 & 3.2 & 1&334e-04 & 2.9 & 4&27e-03 \\
 32 & 1&342e-05 & 3.1 & 3&137e-05 & 3.0 & 7&253e-06 & 3.1 & 1&695e-05 & 3.0 & 1&39e-03 \\
 64 & 1&655e-06 & 3.0 & 3&937e-06 & 3.0 & 8&943e-07 & 3.0 & 2&127e-06 & 3.0 & 6&61e-04 \\
128 & 2&060e-07 & 3.0 & 4&927e-07 & 3.0 & 1&113e-07 & 3.0 & 2&662e-07 & 3.0 & 4&48e-04 \\
256 & 2&564e-08 & 3.0 & 6&160e-08 & 3.0 & 1&385e-08 & 3.0 & 3&328e-08 & 3.0 & 2&51e-04 \\
512 & 3&188e-09 & 3.0 & 7&700e-09 & 3.0 & 1&723e-09 & 3.0 & 4&160e-09 & 3.0 & 3&76e-06

\end{tabular}
\caption{$L^1$- and $L^\infty$-errors and observed convergence order $\nu$
for the first two components of the realizability-limited, piecewise
quadratic reconstruction of $\U(x)$ from \eqref{eq:u-limit-test} with
$\gamma = 10^{-10}$ and moment order $N = 4$.}
\label{tab:limiter-test-10}
\end{table}

However, it has been remarked in \cite{Zhang2012a} that in some pathological
situations this limiter may reduce accuracy to second order.
We can demonstrate this by pushing $\U(x)$ closer to the boundary of
realizability by setting $\gamma = 10^{-11}$.
These results are given in \tabref{tab:limiter-test-11}, where
in the $L^1$ error we still see third-order convergence, but in the
$L^\infty$ error the convergence has degraded to second order.

In further tests, which for brevity we omit here,
we observed similar behavior as we decrease $\gamma$ until we arrive at
$10^{-14}$, which is the value of the tolerance in the realizability limiter
(see the beginning of \secref{sec:SIM}).
At this point convergence is only first order as expected.

\begin{table}
\centering
\begin{tabular}{r r@{.}l c r@{.}l c r@{.}l c r@{.}l c r@{.}l}

 & \multicolumn{6}{c}{$u_0$} & \multicolumn{6}{c}{$u_1$}
 & \multicolumn{2}{c}{ } \\
\cmidrule(r){2-7} \cmidrule(r){8-13}
$\ncells$ & \multicolumn{2}{c}{$E^1_h$} & $\nu$
 & \multicolumn{2}{c}{$E^\infty_h$} & $\nu$
 & \multicolumn{2}{c}{$E^1_h$} & $\nu$
 & \multicolumn{2}{c}{$E^\infty_h$} & $\nu$
 & \multicolumn{2}{c}{$\theta_{\max}$} \\ \midrule
 
  8 & 1&308e-03 & ---  & 1&848e-03 & ---  & 7&066e-04 & ---  & 9&983e-04 & ---  & 1&93e-02 \\
 16 & 1&224e-04 & 3.4 & 2&469e-04 & 2.9 & 6&614e-05 & 3.4 & 1&334e-04 & 2.9 & 8&11e-03 \\
 32 & 1&460e-05 & 3.1 & 3&137e-05 & 3.0 & 7&889e-06 & 3.1 & 1&695e-05 & 3.0 & 5&23e-03 \\
 64 & 1&805e-06 & 3.0 & 7&105e-06 & 2.1 & 9&754e-07 & 3.0 & 3&839e-06 & 2.1 & 4&55e-03 \\
128 & 2&247e-07 & 3.0 & 1&721e-06 & 2.0 & 1&214e-07 & 3.0 & 9&299e-07 & 2.0 & 4&32e-03 \\
256 & 2&799e-08 & 3.0 & 4&162e-07 & 2.0 & 1&512e-08 & 3.0 & 2&248e-07 & 2.0 & 4&15e-03 \\
512 & 3&457e-09 & 3.0 & 8&958e-08 & 2.2 & 1&868e-09 & 3.0 & 4&840e-08 & 2.2 & 3&57e-03

\end{tabular}
\caption{$L^1$- and $L^\infty$-errors and observed convergence order $\nu$
for the first two components of the realizability-limited, piecewise
quadratic reconstruction of $\U(x)$ from \eqref{eq:u-limit-test} with
$\gamma = 10^{-11}$ and moment order $N = 4$.}
\label{tab:limiter-test-11}
\end{table}

That the limiter works so close to the boundary of realizability is
satisfactory to us:
When moment vectors as close as $10^{-10}$ to the boundary of realizability
appear in a simulation, errors from the numerical optimization are likely to
affect convergence before these effects of the realizability limiter play a
role.

\subsection{Plane source}
In this test case we start with an isotropic distribution where the initial 
mass is concentrated in the middle of an infinite domain $x \in
(-\infty, \infty)$:
\begin{align*}
 \psiInit(x, \mu) = \psi_{\rm floor} + \delta(x)
\end{align*}
where the small parameter $\psi_{\rm floor} = 0.5 \times 10^{-8}$ is used to
model a vacuum.%
\footnote{
A vacuum is not exactly realizable by the entropy ansatz \eqref{eq:psiME}
for the Maxwell-Boltzmann entropy.%
}
In practice, a bounded domain must be used, so we choose a domain large
enough that the boundary should have only negligible effects on the
solution:
thus for our final time $\tf = 1$, we take $X = [\xL, \xR] = [-1.2, 1.2]$.
At the boundary we set
\begin{align*}
 \psiL(t, \mu) \equiv \psi_{\rm floor} \quand
 \psiR(t, \mu) \equiv \psi_{\rm floor}
\end{align*}
We use isotropic scattering and no absorption, therefore
$\collision{\psi} = \frac12 \vint{\psi} - \psi$,
$\sig{s} = 1$ and $\sig{a} = 0$.

We approximate the delta function by using an even number of spatial cells
and splitting the delta into the cells immediately to the left and right
of $x = 0$.
This is then projected into $V_h^k$ using \eqref{eq:dweakform1a}.%
\footnote{
We refer the reader to \cite{yang-shu-2013} for a thorough discussion of using
discontinuous Galerkin methods for problems with delta functions in the
initial data.
}
All solutions here are computed with $\ncells = 300$ cells and
spatial polynomials of degree $\order = 2$.

\figref{fig:PlanesourceCuts} presents solutions for different 
moment models.
This figure includes a reference solution, which we computed using our
scheme for the P$_{99}$ model with $\ncells = 2000$ cells and spatial order
$k = 0$.
The oscillations we see have been observed before and arise due to the
fact that we are using moment models (which are indeed spectral methods)
on a non-smooth problem.
Indeed, our M$_{\nmom}$ results agree well with those in
\cite{Hauck2010}.
The full- and mixed-moment models look largely similar for the same numbers
of degrees of freedom, with the notable differences being around
$x = 0$.
Here the mixed-moment solutions are much more sharply peaked while the
full-moment solutions are more flat and wider.
The discrepancy in magnitude, however, seems to decrease as $\nmom$
increases, which agrees with the expectation that both methods
are converging as $\nmom \to \infty$.

\begin{figure}[htbp!]
\centering
 \subfloat[Full-moment models.]{
  \includegraphics[width=0.45\textwidth]{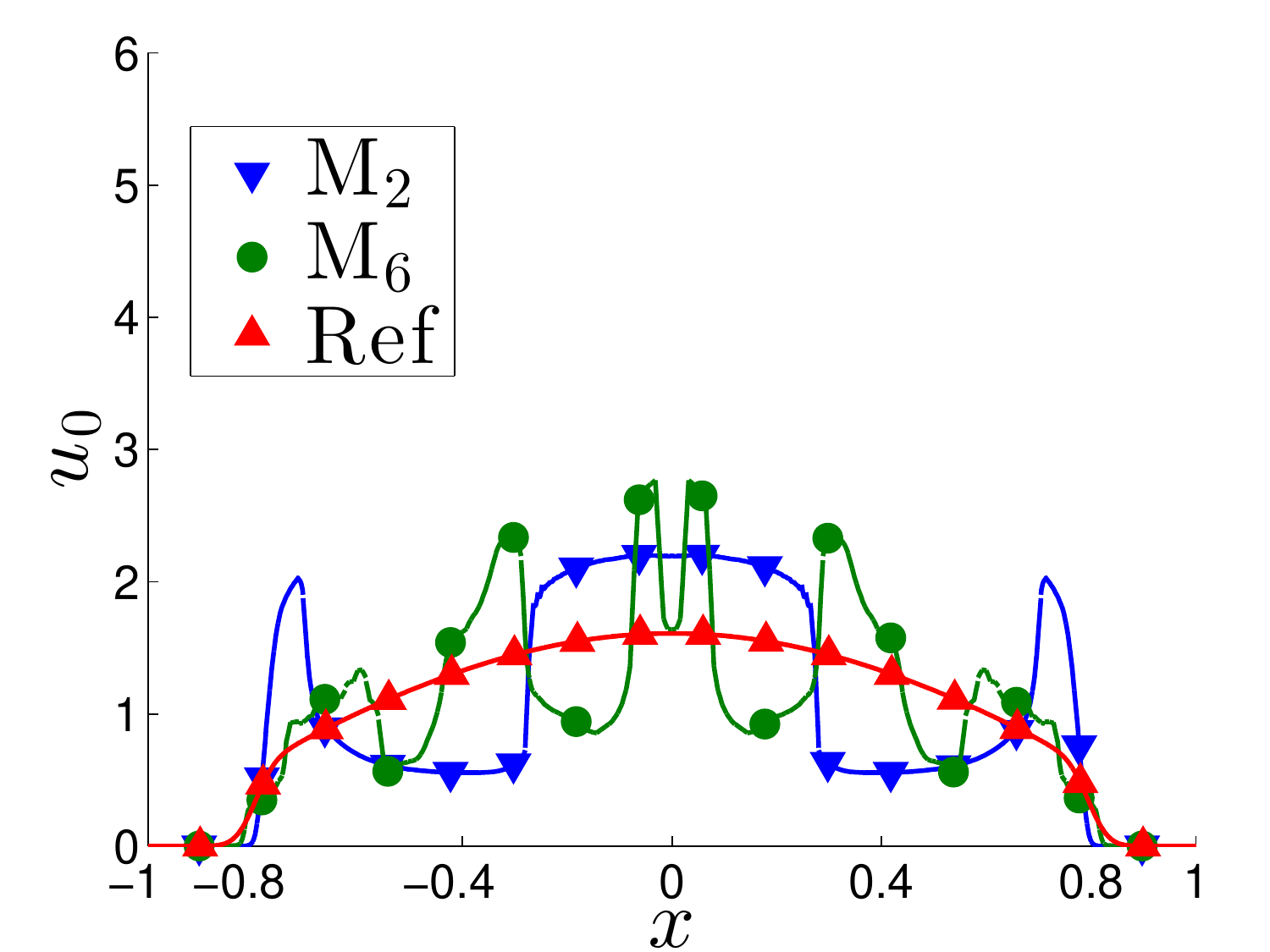}
  }
    \hfill
 \subfloat[Mixed-moment models.]{
  \includegraphics[width=0.45\textwidth]{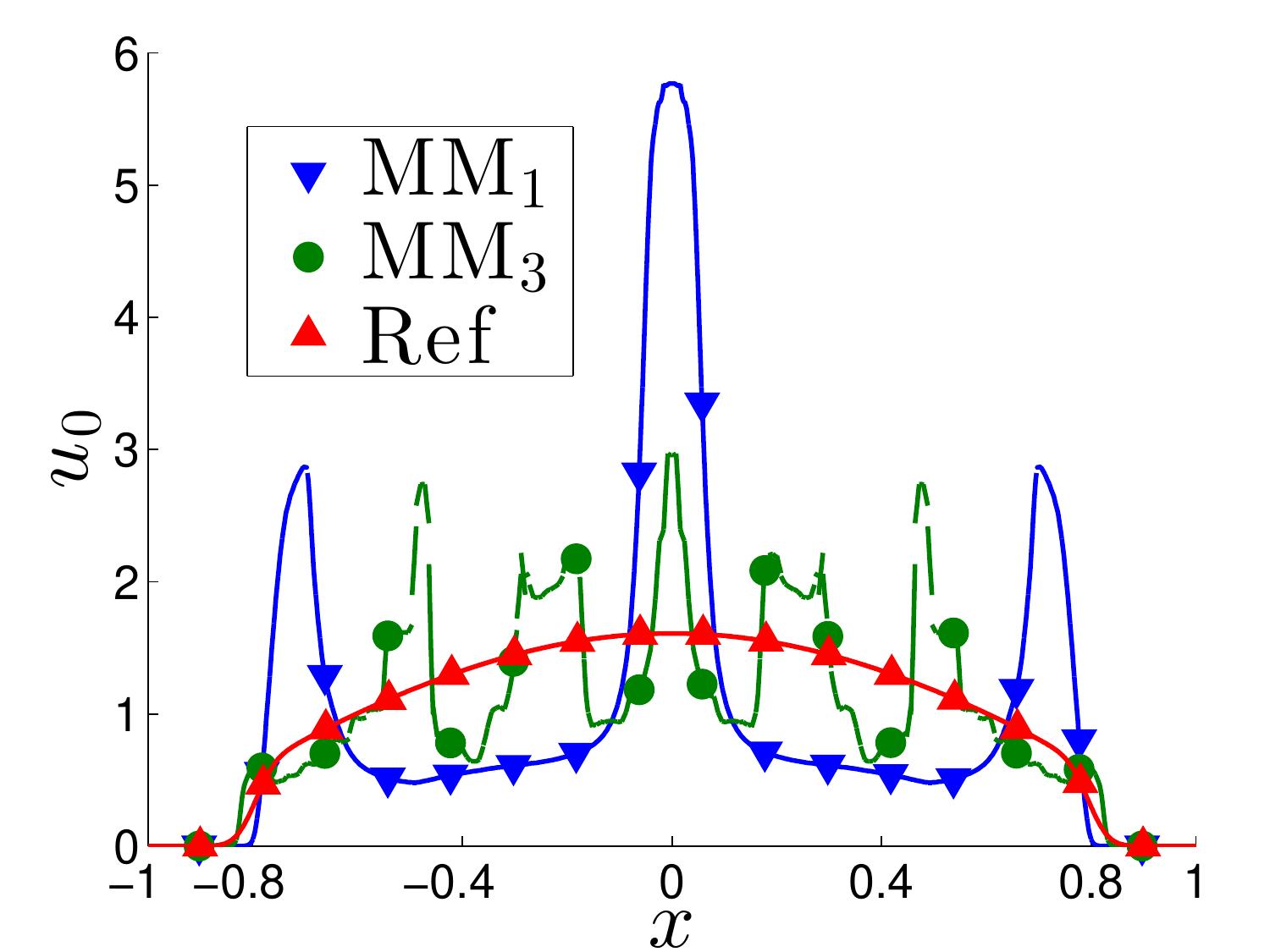}
  }
\caption{Local density $u_0$ for different models and a reference solution
at $t = 0.8$ in the plane-source problem.}
\label{fig:PlanesourceCuts}
\end{figure}

\FloatBarrier
\figref{fig:PlanesourceTheta} shows the activity of the realizability
limiter.
The realizability limiter is most active along the fronts where
particles from the initial impulse are first entering the domain,
which is as expected, because this is where the moment vectors of the
solution lie closest to the boundary of realizability \cite{AllHau12}.
We also see that as the number of moments increases, the activity
of the realizability limiter increases as well.
This is also not surprising, since realizability conditions typically require
tighter and
tighter bounds on the moment components as their order increases.
Thus numerical errors of a similar size will have an increasingly large
chance of pushing the solution out of the realizable set.
We did not observe significant differences between the full-moment and
mixed-moment models.

\begin{figure}[htbp!]
\centering
 \subfloat[M$_3$]
  {\includegraphics[width=0.45\textwidth]{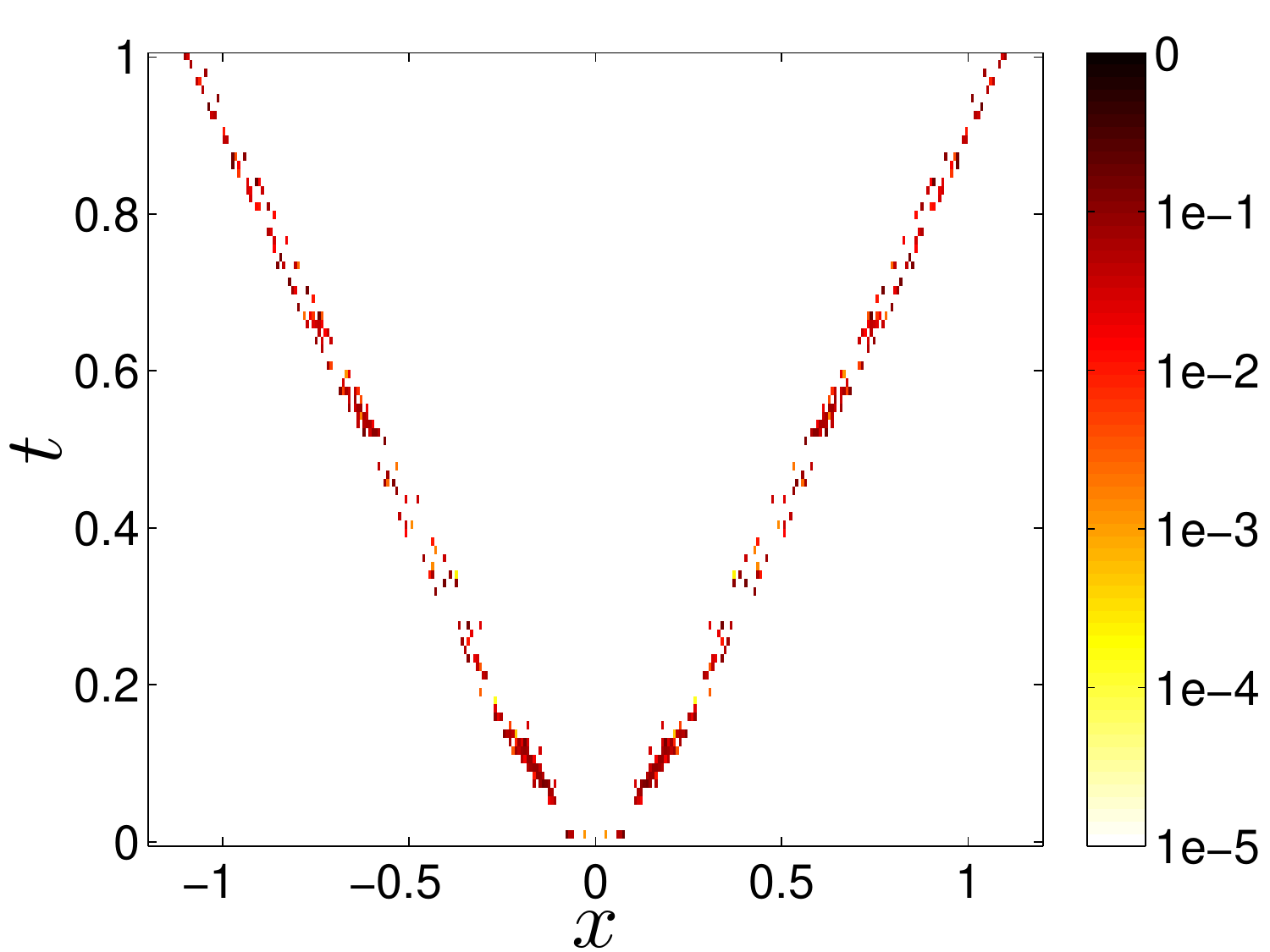}}
  \hfill
 \subfloat[M$_5$]
  {\includegraphics[width=0.45\textwidth]{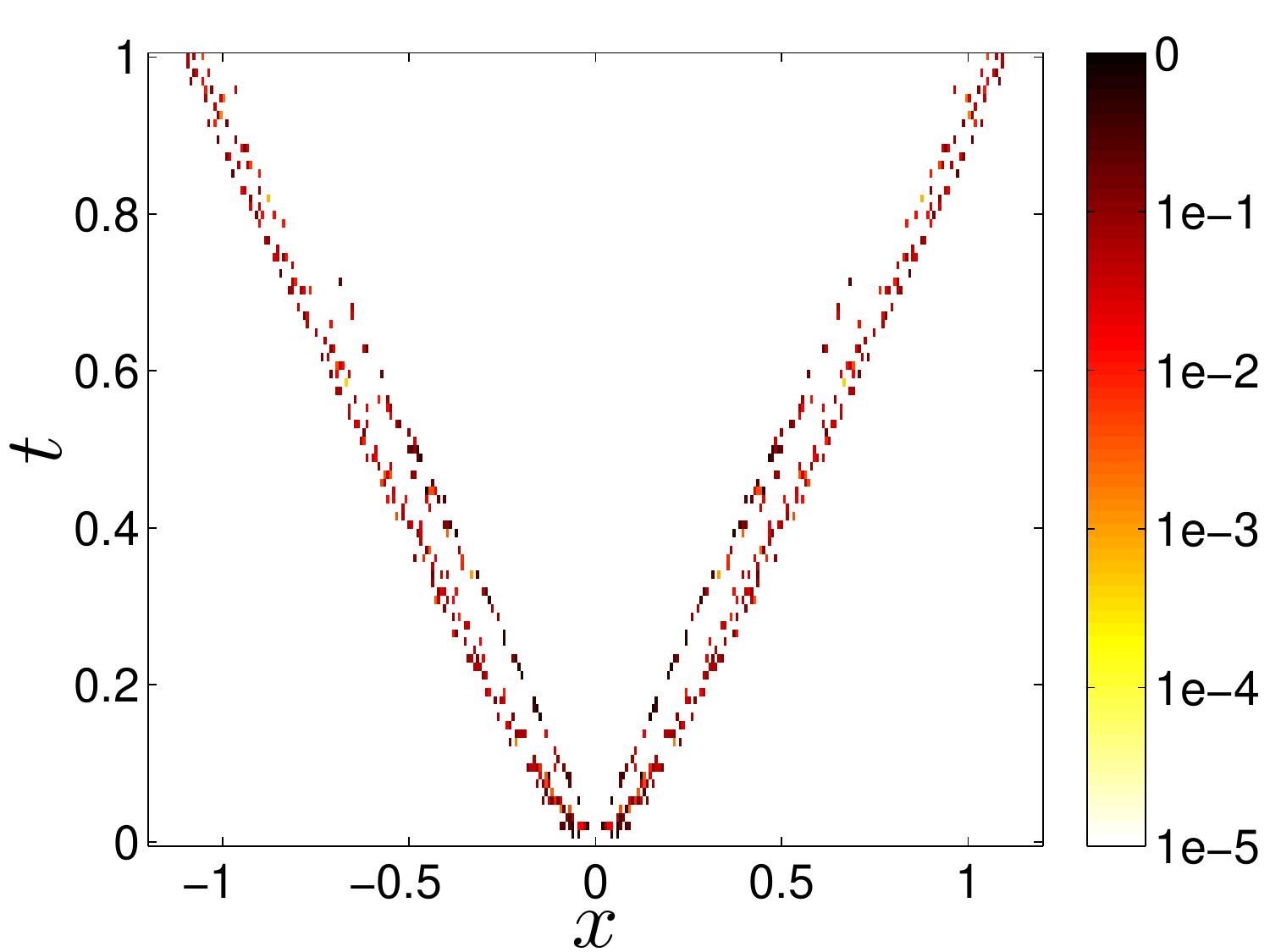}}
\caption{The value of $\theta$ in the realizability limiter for two models
of the plane-source problem.
Note that we choose a logarithmic scale so that even small values of
$\theta$ are noticeable.}
\label{fig:PlanesourceTheta}
\end{figure}

\subsection{Two beams}
This test case models two beams entering a (nearly) empty absorbing medium.
This is a classical test problem used to illustrate the shortcomings of
the M$_1$ model, whose steady-state solution for this problem has a
nonphysical shock.
This test case is also challenging for the numerical optimization
\cite{AllHau12}.

The domain is $X = (-0.5, 0.5)$, and
the beams are specified by boundary conditions
$$
 \psiL(t, \mu) = \frac{1}{\Sigma} \exp \left( -\frac{(\mu - 1)^2}
  {2\Sigma^2} \right)
  \quand
 \psiR(t, \mu) = \frac{1}{\Sigma} \exp \left( -\frac{(\mu + 1)^2}
  {2\Sigma^2} \right),
$$
with $\Sigma = 50$.
The initial condition again approximates a vacuum similarly as before:
\begin{align*}
 \psi(0, x, \mu) = \psi_{\rm floor} = 0.5 \times 10^{-8}.
\end{align*}
Finally, we use absorption parameter $\sigma_a = 2$ and no scattering, 
$\sigma_s = 0$.
Again, all solutions here are computed with $\ncells = 300$ cells and spatial
polynomials of degree $\order = 2$.

Our numerical solutions for the local density $u_0$ all look qualitatively
the same---with the notable exception of the M$_1$ model---and match the
true steady-state solution, so in \figref{fig:beams-cuts-MM2} we only
present one example solution for the unfamiliar reader.
We do note, however, that our steady-state solutions for the full-moment
models do not contain the shocks observed in \cite{Hauck2010}.
This difference is apparently due to the fact that we use sharply
forward-peaked boundary conditions, where as isotropic boundary conditions
were used in \cite{Hauck2010}.
As shown in \figref{fig:beams-cuts-MM2}, the mixed-moment solutions do not
appear to contain any steady-state shocks either.

In \figref{fig:beams-cuts-M4} we present a zoomed-in view of the
oscillations present in a typical transient solution.
These oscillations become more noticeable for higher-order models
(M$_{\nmom}$ for $\nmom \ge 4$ and MM$_{\nmom}$ for $\nmom \ge 3$).
It seems clear to us that these are numerical artifacts, and we
believe they would be mitigated with a more sophisticated slope
limiter.

\begin{figure}[htbp!]
\centering
 \subfloat[MM$_2$]
  {\includegraphics[width=0.45\textwidth]{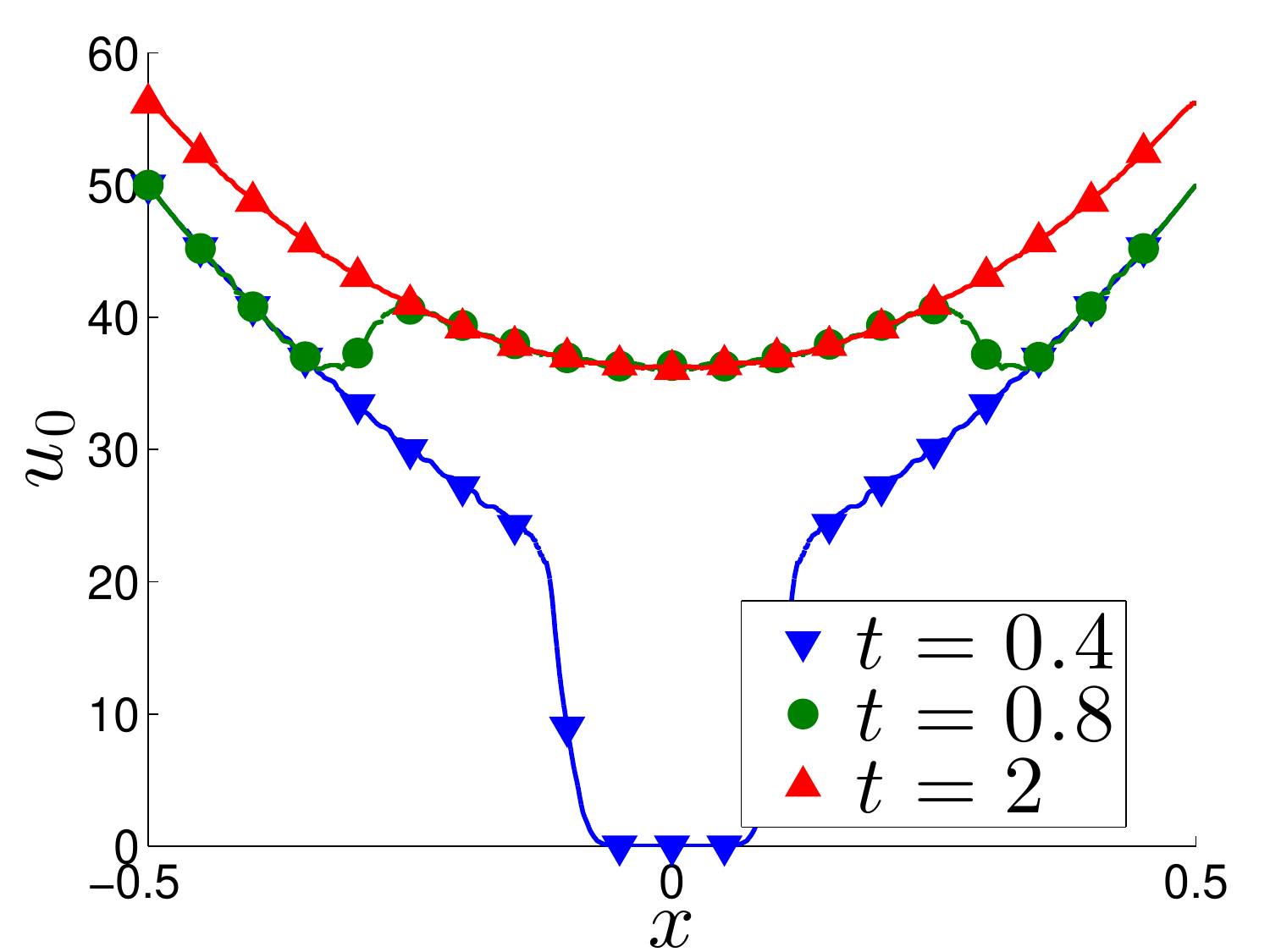}
   \label{fig:beams-cuts-MM2}}
 \hfill
 \subfloat[The M$_4$ model at $t = 0.8$.]
  {\includegraphics[width=0.45\textwidth]{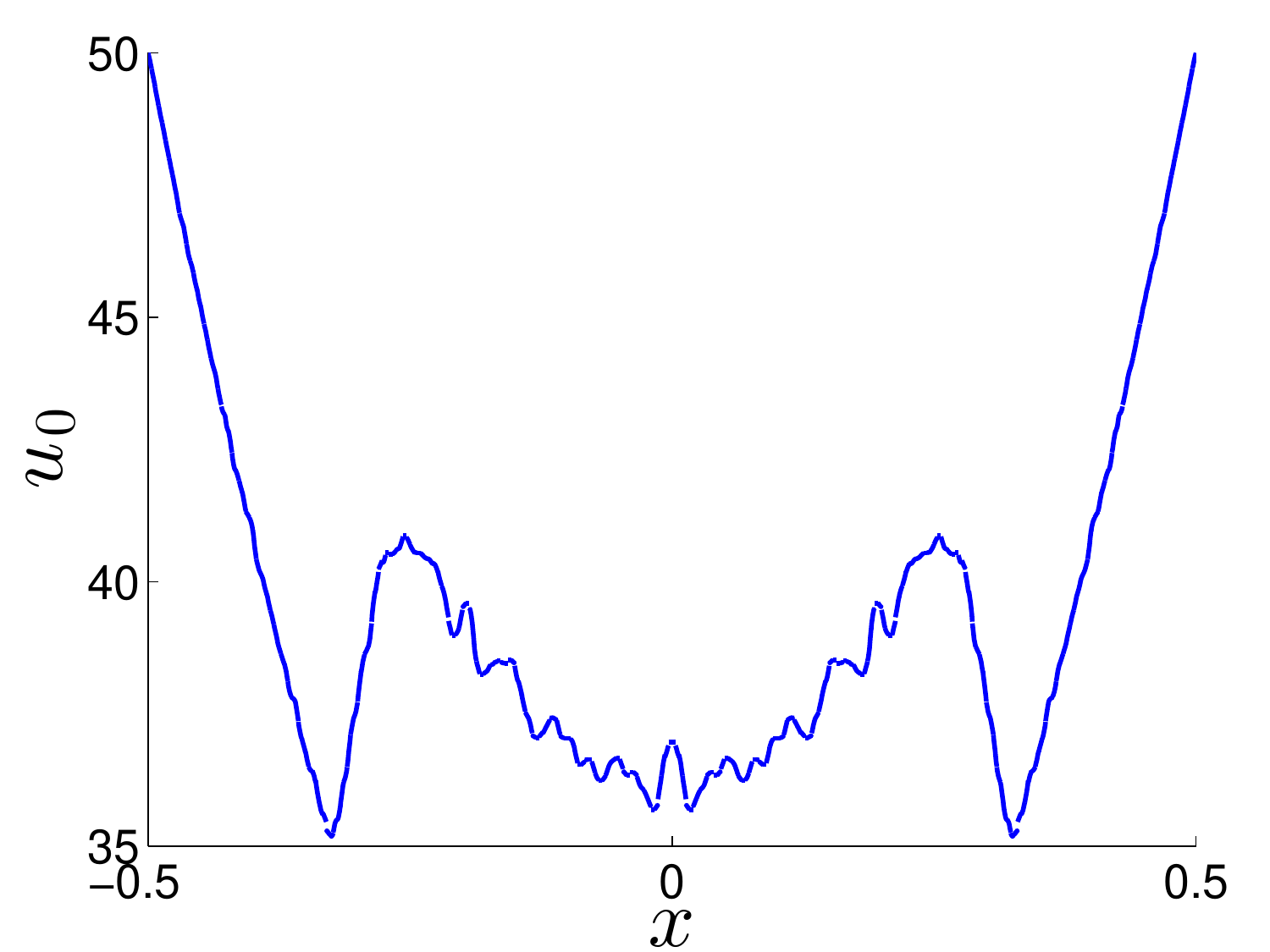}
   \label{fig:beams-cuts-M4}}
\caption{Local density $u_0$ for different models at $t = 0.8$ in the
two-beam problem.}
\label{fig:BeamsCuts}
\end{figure}
 
The activity of the realizability limiter again increases with the number
of moments, but in this problem we see differences between full- and
mixed-moment models.
\figref{fig:BeamsTheta} illustrates this difference.
The reason for this difference is not yet clear to us, but it seems to
indicate that the mixed-moment model is converging more slowly to
steady state.

\begin{figure}[htbp!]
\centering
 \subfloat[M$_4$]
  {\includegraphics[width=0.45\textwidth]{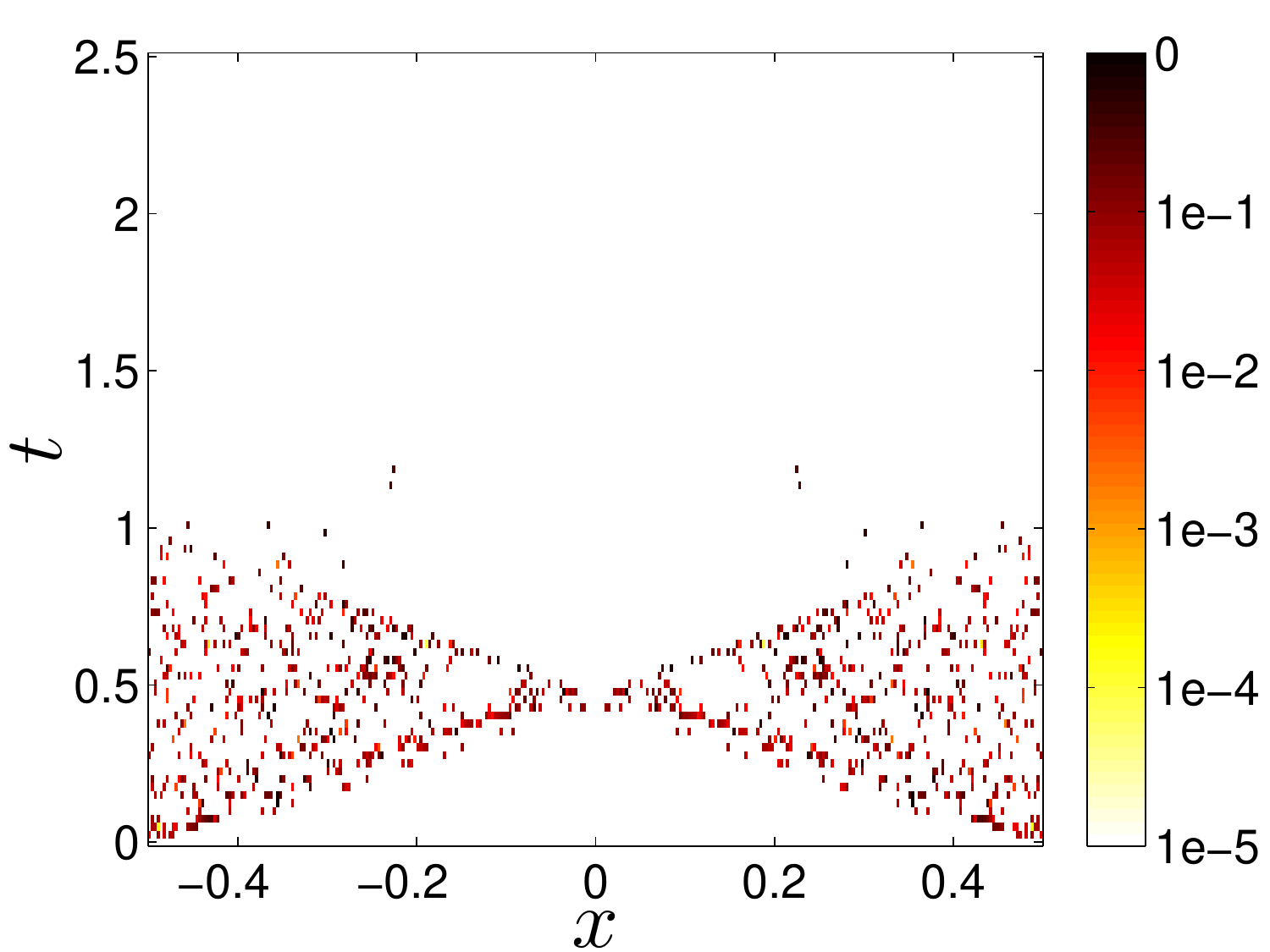}}
 \hfill
 \subfloat[MM$_2$]
  {\includegraphics[width=0.45\textwidth]{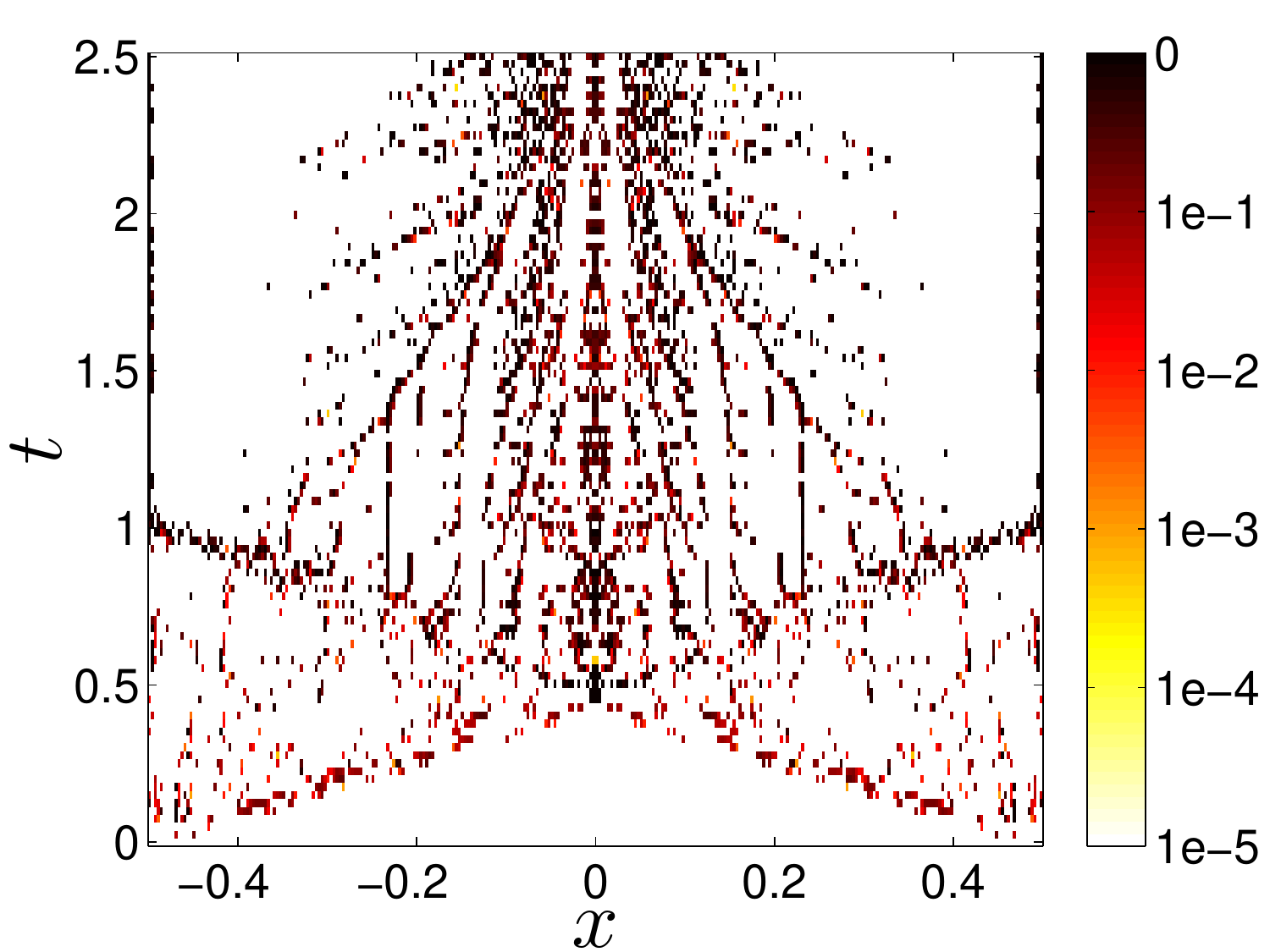}}
\caption{The value of $\theta$ in the realizability limiter for two models
of the two-beam problem.
Note that we choose a logarithmic scale so that even small values of
$\theta$ are noticeable.}
\label{fig:BeamsTheta}
\end{figure}

\section{Conclusions and outlook}
\label{sec:Conclusions}
We presented a high-order Runge-Kutta discontinuous Galerkin scheme for
minimum-entropy moment models of linear kinetic equations in one space
dimension.
The key issue for higher-order methods for minimum-entropy moment models is
that the numerical solution typically leaves the set of realizable moments,
even though standard techniques can be used to show that the cell means of the
solution remain realizable.
We address this problem using a realizability limiter inspired by the
positivity-preserving limiter used in \cite{Zhang2010} for the Euler equations.
Such a limiter requires the computation of the intersection of a line in
moment space with the boundary of the realizable set, a set which typically
has nonlinear boundaries.
We are able to approximate this intersection by replacing the true
realizable set with its quadrature-based approximation, which is a
convex polytope.
This quadrature-based approximation is intriguing because it is a convex
polytope for any moment order and any dimension of the angular domain
indicating that our techniques could be extended to these cases.

We constructed a new manufactured solution whose source term is realizable, thus
allowing us to consider target solutions closer to the boundary of the
realizable set.
These tests show that our scheme converges
as expected and that higher-order schemes are more efficient.
We also present numerical solutions for standard benchmark problems, where we
are able to compare full- and mixed-moment models.

Future work should focus on a parallelized implementation for two- and
three-dimensional problems.
Theoretically, implementation of the quadrature-based realizability limiter
requires no change because the convex polytopic structure of $\RQ{\basis}$
holds in any dimension.
Practically speaking, however the main challenge is that the number of facets
grows quickly with the number of moments and number of quadrature points, both
of which will be higher.
Further work in higher-order methods will also have to consider new methods
for time integration, as here we relied heavily on the SSP property, which
is not possible past fourth-order.
Relatedly, at least partially implicit time integrators should be investigated,
particularly in the context of constructing an asymptotic preserving numerical
method for the moment system.


\appendix



\section{The number of facets in $\RQone{\mbasis}$ and $\RQone{\mmbasis}$}

Even with some speed-ups in the computation of the facet-intersections and possibly
approximations by removing facets, the number of facets plays a large role in
determining the complexity of finding the intersection of a ray with
the boundary of the
convex polytope $\RQsone{\basis}$.
In this section we mention how some results from the study of convex polytopes
give the exact number of facets in the full-moment case, and then we compare this
with an upper bound of the number of facets in the mixed-moment case.

First, some notational remarks for this section:
For convenience, we work with the closures $\RQblOne$ and $\RQbOne$
of $\RQsone{\basis}$ and $\RQone{\basis}$ respectively.
When working with $\RQbOne$, we consider it as a subset of $\bbR^N$ (or $\bbR^{2N}$
in the mixed-moment case), and use the notation $\bu_1$ and $\basis_1$ to indicate
the final $N$ (or $2N$) entries of $\bu$ and $\basis$ respectively.
We also often work with the matrix form of the half-space representation,
so for example $\RQbOne = \{ \bu_1 : A \bu_1 \le b \}$, for a matrix $A \in
\bbR^{d \times N}$ with rows $\{\ba^T_i\}_{i = 1}^d$ and a vector $b \in \bbR^d$.
Finally, we omit proofs in this section because we consider the arguments needed
to be unenlightening and relatively straightforward.

We first note that the number of facets of $\RQblOne$ is only one more than that of
$\RQbOne$:
\begin{proposition}
If $A$ and $b$ define a half-space representation of $\RQbOne$, then a half-space
representation of $\RQblOne$ is:
$$
 \RQblOneBig = \left\{ \bu = \begin{pmatrix} u_0 \\ \bu_1 \end{pmatrix}
  : \begin{pmatrix} 1 & 0 \\ -b & A \end{pmatrix}
  \begin{pmatrix} u_0 \\ \bu_1 \end{pmatrix}
  \le \begin{pmatrix} 1 \\ 0 \end{pmatrix}
  \right\}.
$$
\end{proposition}
Therefore in the sequel we focus on the number of facets of $\RQbOne$.

First we consider the full-moment case.
The convex polytope $\RQmOne \subset \bbR^N$ 
is known as the \textit{cyclic polytope} and plays a special role in the
study of convex polytopes.
The Upper Bound Theorem states that for a given number of vertices in a given
dimension, the cyclic polytope has the maximum number of facets
\cite{brondsted-convex-polytopes}.
Gale's evenness condition or
the Dehn-Sommerville equations can be used to show that the number of
facets is
$$
 C(\nmom, \nqmu) = \begin{pmatrix}\nqmu  - \floor{\frac12 (\nmom + 1)} \\
  \nqmu - \nmom \end{pmatrix}
  + \begin{pmatrix} \nqmu - \floor{\frac12 (\nmom + 2)} \\ \nqmu - \nmom \end{pmatrix}
$$
for $\nqmu > \nmom > 1$ \cite{brondsted-convex-polytopes}, where
$\floor{\cdot}$ indicates the integer part of its argument.
We note that this holds for any choice of distinct quadrature nodes
$\{ \mu_i \}$.
Since $\RQmOne$ has $C(\nmom, \nqmu)$
facets, there exists a half-space representation such that $A \in
\bbR^{C(\nmom,
\nqmu) \times \nmom}$ and $b \in \bbR^{C(\nmom, \nqmu)}$.
Unpacking the definition of the binomial coefficient we can see that for
fixed, even $N$, we have $C(\nmom, \nqmu) = \cO(\nqmu^{\nmom/2})$, and for
fixed, odd $\nmom$ we
have $C(\nmom, \nqmu) = \cO(\nqmu^{(\nmom - 1)/2})$.

One can see from \figref{fig:Facets} that $\MMN$ models appear always to have
fewer facets than the corresponding $\MN$ models.
To show that this holds more generally, we first need
a half-space representation for $\RQmmOne$.
This representation can be derived using the half-space representations
from the full-moment case.

\begin{proposition}
Let $A_\pm$ and $b_\pm$ define half-space representations for the convex polytopes
formed by the basis functions on the positive and negative subintervals
respectively:
\begin{align*}
 \co \left\{\mbasis_1(\mu_i) \right \}_{\mu_i \ge 0}
  &= \left\{ \bu_{1+} : A_+ \bu_{1+} \le b_+ \right\}, \\
 \co \left\{\mbasis_1(\mu_i) \right \}_{\mu_i \le 0}
  &= \left\{ \bu_{1-} : A_- \bu_{1-} \le b_- \right\}.
\end{align*}
We assume $b_\pm \ge 0$ component-wise.%
\footnote{%
Here we are using the fact that the subintervals for the mixed-moments are joined
exactly at $\mu = 0$ and assume furthermore that this point is a quadrature node.
This is indeed a reasonable assumption, since even in MM$_1$, a delta function can
form at $\mu = 0$.
}
Then a half-space representation for $\RQmmOne$ is given by
$$
 \RQmmOneBig = \left\{ \bu_1 = \begin{pmatrix}\bu_{1+} \\ \bu_{1-} \end{pmatrix}
  : A \bu_1 \le b \right\},
$$
where
$$
 A = \begin{pmatrix} A_+ & 0 \\ 0 & A_- \\
  \vdots & \vdots \\
  b^{-1}_{+i} \ba_{+i}^T & b^{-1}_{-j} \ba_{-j}^T \\
  \vdots & \vdots
  \end{pmatrix}
  \quand
  b = \begin{pmatrix} b_+ \\ b_- \\ \vdots \\ 1 \\ \vdots \end{pmatrix};
$$
where the last rows of the $A$ include only those pairs $\{ i, j \}$ such that
neither $b_{+i}$ nor $b_{-j}$ is equal to zero.
\end{proposition}

If we let $C_\pm$ denote the number of rows of $A_\pm$ respectively,
this representation gives $C_+ + C_- + C_+C_-$ as an upper-bound on the number of
facets in $\RQmmOne$.%
\footnote{%
A consideration of the most basic case, MM$_1$, shows that indeed some of the
inequalities in this half-space representation are redundant, but at the moment we
are unable to say in general exactly how many are redundant.
}
The number of rows such that $b_{\pm i} = 0$ is equal to the number of facets
including the vertex corresponding to the quadrature point at $\mu = 0$.
These facets can be more generally described as those containing the vertex
corresponding to the first quadrature point, when the quadrature points are
arranged in increasing order.
The number of such facets can be computed using Gale's evenness condition (see
Theorem 13.6 and Exercise 13.1 in \cite{brondsted-convex-polytopes}).
We omit this computation here but note that removing these facets does not change
the order of the number of facets (nor any relevant leading-order coefficients), so
in the comparison that follows, we ignore these terms.

To compare the full-moment and mixed-moment cases for the same number of degrees of
freedom, one would consider the full-moment case
of order $N$, for $N$ even, and the mixed-moment case of order $N/2$.
Let us assume that we use a quadrature set which includes $\mu = 0$ and has $Q / 2$
points over both $\mu \ge 0$ and $\mu \le 0$, for a total of $Q - 1$ points (since
the point at $\mu = 0$ should not be counted twice in the full-moment case).
Then the number of facets in the full-moment case is $C(N, Q - 1)$ while in the
mixed-moment case, the number of facets in our half-space representation is on the 
order of $C(N / 2, Q / 2)^2$.
Then, straightforward calculations show that
when $N / 2$ is odd we have $C(N / 2, Q / 2)^2 = \cO(Q^{N / 2 - 1})$, which is one
order less than in the full-moment case.
When $N / 2$ is even, our half-space representation for the mixed-moment case has
$\cO(Q^{N / 2})$ facets, which is
the same order as the full-moment case.
However, the leading-order coefficient is smaller in the mixed-moment case,
thereby showing that the number of facets in the mixed-moment case is at
least asymptotically smaller.
Indeed, if we let $N = 4n$, the ratio of the highest-order coefficients is
$(2n)!/(2^n n!)^2$, which is
bounded by $1/2$ and monotonically decreases with $n$.

\longpaper{
\subsection{Jacobians of minimum-entropy fluxes}
\label{sec:Jacobians}
For full moments minimum-entropy one can easily see that the Jacobian is given 
by
\begin{align*}
\begin{pmatrix}
0&1&0&\cdots&0\\
0&0&1&\cdots&0\\
\vdots&&\ddots&\vdots\\
0&0&0&\cdots&1\\
&&\left(\nabla_\U u_{n+1}\right)^T &&
\end{pmatrix} &&\text{ with }&
\begin{pmatrix}
u_{0}&\cdots&u_{n}\\
\vdots & \ddots & \vdots\\
u_{n} & \cdots & u_{2n}
\end{pmatrix}\nabla_uu_{n+1} = \begin{pmatrix}
u_{n+1}\\\vdots\\u_{2n+1}\end{pmatrix}
\end{align*}
Similarly one can conclude the Jacobian for mixed moment minimum-entropy 
methods:
\begin{align*}
\begin{pmatrix}
0&1&1&0&\cdots&0\\
0&0&0&1&\cdots&0\\
\vdots&&\ddots&\vdots\\
0&0&0&0&\cdots&1\\
&&\left(\nabla_\U \psip{k}\right)^T &&\\
&&\left(\nabla_\U \psim{k}\right)^T &&
\end{pmatrix}
\end{align*}
with
\small
\begin{align*}
\begin{pmatrix}
u_{0}&\psip{1}&\psim{1}&\psip{2}&\psim{2}&\cdots&\psip{n}&\psim{n}\\
\psip{1}&\psip{2}& 0& \psip{3}&0&\cdots&\psip{n+1}&0\\ 
\psim{1}& 0&\psim{2}& 0& \psim{3}&\cdots& 0&\psim{n+1}\\
\vdots&&&&\vdots&&&\vdots\\
\psip{n}&\psip{n+1}& 0& \psip{n+2}&0&\cdots&\psip{2n}&0\\ 
\psim{n}& 0&\psim{n+1}& 0& \psim{n+2}&\cdots& 0&\psim{2n}
\end{pmatrix}\nabla_u \psip{k} = 
\begin{pmatrix}
\psip{k}\\\psip{k+1}\\0\\\vdots\\\psip{k+n}\\0
\end{pmatrix}\\
\begin{pmatrix}
u_{0}&\psip{1}&\psim{1}&\psip{2}&\psim{2}&\cdots&\psip{n}&\psim{n}\\
\psip{1}&\psip{2}& 0& \psip{3}&0&\cdots&\psip{n+1}&0\\ 
\psim{1}& 0&\psim{2}& 0& \psim{3}&\cdots& 0&\psim{n+1}\\
\vdots&&&&\vdots&&&\vdots\\
\psip{n}&\psip{n+1}& 0& \psip{n+2}&0&\cdots&\psip{2n}&0\\ 
\psim{n}& 0&\psim{n+1}& 0& \psim{n+2}&\cdots& 0&\psim{2n}
\end{pmatrix}\nabla_u \psim{k} = 
\begin{pmatrix}
\psim{k}\\0\\\psim{k+1}\\\vdots\\0\\\psim{k+n}
\end{pmatrix}
\end{align*}
}{}
\bibliographystyle{plain}
\bibliography{RadLit}

\end{document}